\documentclass[final]{siamltex}

\usepackage{graphicx, color}
\usepackage{float}

\usepackage{multirow}
\usepackage{booktabs}
\usepackage{kbordermatrix}
\usepackage{rotating}

\usepackage{tikz}
\usetikzlibrary{arrows,shapes,snakes,automata,backgrounds,petri}
\usetikzlibrary{graphs}

\newcommand{\revised}[1]{{#1}}

\usepackage[pagewise, display math, mathlines]{lineno}

\usepackage{amsmath, amssymb}

\newcommand{\ignore}[1]{}

\newcommand\figref{Figure~\ref}
\newcommand\tabref{Table~\ref}

\allowdisplaybreaks 

\usepackage{pdflscape}
\usepackage{hyperref}
\usepackage{setspace}
\usepackage[ruled,vlined, linesnumbered]{algorithm2e}

\newtheorem{example}{Example}[section]

\SetKwFor{For}{for}{do}{endfor}

\title{Chance-Constrained Combinatorial Optimization with a Probability Oracle and Its Application to Probabilistic Partial Set Covering
\thanks{	Simge K\"{u}\c{c}\"{u}kyavuz and Hao-Hsiang Wu have been supported, in part, by National Science Foundation Grants 1732364 and 1733001.}
}

\author{Hao-Hsiang Wu \footnotemark[2] \and S\.imge K\"u\c{c}\"ukyavuz \footnotemark[2] \footnotemark[3]}

\begin{document}

\maketitle

\renewcommand{\thefootnote}{\fnsymbol{footnote}}

\footnotetext[2]{Department of Industrial and Systems Engineering,
	University of Washington, Seattle, WA ({\tt hhwu2@uw.edu}, {\tt simge@uw.edu})}
\footnotetext[3]{Corresponding author.}

\renewcommand{\thefootnote}{\arabic{footnote}}

\begin{abstract}
 		We investigate a class of  chance-constrained combinatorial optimization problems.  Given a pre-specified risk level $\epsilon \in [0,1]$, the chance-constrained program aims to find the minimum cost selection of a vector of binary decisions  $x$
such that a desirable event $\mathcal{B}(x)$ occurs with probability at least 	
$ 1-\epsilon$. In this paper, we assume that we have an oracle that computes $\mathbb P( \mathcal{B}(x))$ exactly. Using this oracle, we propose a general exact method for solving the chance-constrained problem. 
In addition, we show that 
if the chance-constrained program is solved approximately by a sampling-based approach, then the oracle can be used as a tool for checking and fixing the feasibility of the  solution given by this approach. We demonstrate the effectiveness of our proposed methods on a probabilistic partial set covering problem (PPSC), which admits an efficient probability oracle. We give a compact mixed-integer program that  solves PPSC optimally (without sampling) for a special case.  For large-scale instances for which the exact methods exhibit slow convergence, we propose a sampling-based approach that exploits the special structure of PPSC. In particular, we introduce a  new class of facet-defining inequalities for a submodular substructure of PPSC, and show that a sampling-based algorithm coupled with the probability oracle provides high-quality feasible solutions to the large-scale test instances effectively.
\end{abstract}


\pagestyle{myheadings}
\thispagestyle{plain}
\setstretch{1.06}

\section{Introduction}
	Chance-constrained programs (CCPs), first introduced in \cite{Charnes1958}, aim to find the optimal solution to a problem such that the probability of satisfying certain constraints is at least at a certain confidence level.	
In this paper, we consider chance-constrained {\it combinatorial} optimization problems.
Given a vector of $n$ binary decision variables $x \in \mathbb{B}^n$, we define $x_i=1$ if the $i^{th}$ component of $x$ is selected, $x_i=0$ otherwise. Let $\mathcal{B}(x)$ represent a random event of interest for a given $x$. Given a risk level $\epsilon \in [0,1]$, a chance-constrained program is 
\begin{linenomath}
\begin{equation}\label{model:ori_cc}
	\min\{b ^\top x : \mathbb P( \mathcal{B}(x))\ge 1-\epsilon, x \in \mathcal{X} \cap \mathbb{B}^n \},
\end{equation} 
\end{linenomath}
where $b \in \mathbb{R}^n$ is a given cost vector,  the set $\mathcal{X}$ represents the deterministic constraints on the variables $x$, and $\mathbb P( \mathcal{B}(x)) \ge 1-\epsilon $ represents the restriction that the probability of event $\mathcal{B}(x)$  must be at least $1-\epsilon$. There are three sources of difficulty for this class of problems. First,  for a given $x$, computing $\mathbb P( \mathcal{B}(x))$ exactly is hard in general, because it involves multi-dimensional integrals. Second, the feasible region of chance-constrained programs is non-convex for general probability distributions. Finally,  due to the combinatorial nature of the decisions, the search space is very large. 

In the CCP literature, the first challenge of evaluating the probability of an event, $\mathbb P( \mathcal{B}(x))$, is generally overcome by sampling from the true distribution \cite{Prekopa1990,Ruszczynski2002,Beraldi2010,Nemirovski2006,Luedtke2008}. This creates an approximation of the chance constraint, which can be evaluated for the given samples. 	In contrast, in this paper, we assume that there exists an efficient oracle, which provides an exact value of $\mathbb P( \mathcal{B}(x))$ for a given $x$.  We give a delayed cut generation algorithm to solve  problem \eqref{model:ori_cc} exactly using the true distribution, instead of sampling from the true distribution. We show that  probabilistic partial set covering problems (PPSC) under certain distributions admit an efficient probability oracle. Using this class of problems  in our computational study, we demonstrate the effectiveness of the proposed algorithm for small-size problems.  However, due to the exponential decision space, convergence may be slow for larger problems. We observe that the explicit (linear) formulation of a chance constraint using a sampling-based approach may be more amenable to exploiting the special structure of the underlying problem. Hence the sampling-based approach may be more effective in solving larger problems. 
However, the solutions to the sample approximation problem may not satisfy the chance constraint under the true distribution if not enough samples are used. On the other hand, the solution methods may be slow for large sample sizes. 
For such a sampling-based approach, we propose a method that utilizes the oracle  to check and correct the feasibility of the approximate solution \revised{by adding {\it globally valid} feasibility cuts to the sample approximation problem. Note that for structured problems, existing sampling-based approaches also add feasibility cuts exploiting the structure of the problem (see, e.g., \cite{Song2014,Minjiao,LK2018}). However, such cuts  for the sample approximation problem are valid based on the scenarios generated, but they may not be globally valid with respect to the original problem under the true distribution.)}

We handle the second difficulty (non-convexity of the feasible region) by expressing the feasibility condition using linear constraints. For the sample approximation problem, such a linear reformulation using additional binary variables is well known \cite{Luedtke2010}.  For the general case (without sampling), we consider a reformulation with exponentially many linear inequalities. We solve this formulation using a 
delayed cut generation algorithm, which starts with a subset of the inequalities, and  adds the violated  inequalities as needed to  cut off the infeasible solutions until a feasible and optimal solution is found. In addition, for a special case,  we  show that there exists a compact (polynomial-size) mixed integer linear program (MIP) that solves the problem without the need for sampling.    To handle the third difficulty, we show that we can use the properties of the oracle to obtain stronger inequalities to  represent the feasibility conditions, which, in turn, reduce the search space of problem \eqref{model:ori_cc} significantly.

\revised{Under certain conditions, if the finite dimensional distributions of the uncertain parameters are log-concave probability measures, then the continuous relaxation of the chance constraint is convex \cite{Prekopa1973}.}	van Ackooij et al.\ \cite{Ackooij2016} consider such convex 
chance-constrained combinatorial optimization problems, where the objective function is non-differentiable. The authors use a sampling (scenario)-based approach and introduce additional binary variables to represent whether the chance constraint is satisfied under each scenario.   They  propose a Benders decomposition algorithm using combinatorial Benders (no-good) cuts, where they use an inexact oracle to approximate the non-differentiable objective value. In contrast, we assume that the objective function is smooth (linear) and use an exact  oracle for evaluating the {\it non-convex} chance constraint.   In another line of work,  van Ackooij and Sagastiz\'{a}bal 
\cite{Ackooij2014} consider CCPs, where  the chance constraint  is convex but hard to evaluate exactly, and the additional  constraints on the  decision variables form a convex set (as a result, the decision variables are continuous). The authors  give a non-smooth optimization (bundle) method  that uses an inexact oracle to evaluate the chance constraint to find an approximate solution.  In contrast, in our problem \eqref{model:ori_cc}, we do not assume convexity of the chance constraint $\mathbb P( \mathcal{B}(x))\ge 1-\epsilon$, or the continuity of  the decision variables,  the binary restrictions on the decision variables form a non-convex set. 

\ignore{
In this paper, we investigate a general method for solving problem \eqref{model:ori_cc} exactly. We assume that there exists an efficient exact oracle, which computes the term $\mathbb P( \mathcal{B}(x))$ exactly for any $x$. An oracle can be a black box or a formulable function for a given $x$.  We introduce a  delayed constraint generation method for obtaining an exact optimal solution to problem \eqref{model:ori_cc}. In this approach,  we solve a relaxation of \eqref{model:ori_cc} and use an oracle to check the feasibility of the resulting solution $x$, and in case $x$ is infeasible, we add a feasibility cut that eliminates this solution. 
}

We demonstrate our proposed methods on a probabilistic partial set covering problem (PPSC) introduced in \cite{ZCAWZ14} for bipartite social networks.
Given a collection of $n$ subsets of $m$ items, a deterministic  set covering problem aims to choose subsets among the collection at a minimum cost, such that each item is covered by at least one chosen subset. In the probabilistic version of this problem we consider, it is assumed that when a subset is chosen, there is uncertainty in which items in the subset are actually covered.  Given a fixed target $\tau \le m$, the probabilistic {\it partial} set covering problem (PPSC) aims to find the minimum cost selection of subsets, which cover at least the target number of items, $\tau$, with  probability $1-\epsilon$. (Note that for $\tau=m$, this problem is equivalent to probabilistic set covering.) Under certain distributions of the random variables,  there exists a polynomial-time oracle to check the feasibility of a given selection of subsets. Using this oracle, we give an exact delayed cut generation algorithm to find the optimal solution. This is equivalent to solving an exponential-sized integer linear program, where an efficient separation algorithm is available. In addition, we show that for a special case of interest, the oracle is formulable and it can be incorporated into the optimization model, which results in a polynomial-sized mixed-integer linear program. While both of these approaches  find optimal solutions to moderate-size problems, the solution times grow exponentially as the problem size increases. In such cases, we develop a modified sampling-based method for PPSC that is able to exploit the special structure of the problem, namely the submodularity. We derive a new class of  valid inequalities  for PPSC that subsumes the submodular inequalities of  Nemhauser and Wolsey \cite{NW88}, and provide conditions under which the proposed inequalities are  facet defining. We observe that the modified sampling-based method is highly effective when combined with the probability oracle to obtain feasible solutions of good quality.	
The literature review on probabilistic  set covering problems and further discussions on alternative approaches are given in the corresponding section (Section  \ref{sec:problem}).

We summarize our contributions and give an outline of the paper as follows. In Section \ref{sec:oracle}, we introduce the concept of a probability oracle for a class of combinatorial CCPs and propose  a general method to solve such CCPs that uses the concept of no-good cuts. We  strengthen the no-good cuts by using the monotonicity of the probability function and the availability of the oracle.   In Section \ref{sec:problem}, we use a class of NP-hard problems (PPSC) to demonstrate the proposed method. In addition, we show that we can solve PPSC by using a compact deterministic MIP under a special case. Furthermore, we propose a modified sampling-based method for PPSC that utilizes its submodular substructure. We  introduce a new class of facet-defining inequalities for this substructure of PPSC.
In addition, we show that  an efficient oracle can be a useful tool for checking and correcting the  feasibility of a solution given by a sampling-based approach.  Furthermore, we propose a modified sampling-based method for PPSC that provides a high-quality feasible solution to the true problem. We  introduce a new class of facet-defining inequalities for the submodular substructure of PPSC that subsumes the known submodular inequalities. We show that we can solve the sample approximation problem of PPSC by using a compact deterministic MIP under a special case.  In Section \ref{sec:computation}, we report the computational results with these alternative approaches. \revised{ Finally, we give a summary of our results and future research directions in Section \ref{sec:con}.} 


		\section{Chance-Constrained Combinatorial Optimization  with a Probability Oracle}\label{sec:oracle} 

Suppose that we have an oracle $\mathcal{A}(x)$, which computes $\mathbb P( \mathcal{B}(x))$ exactly for a given $x$ in polynomial time. We reformulate problem \eqref{model:ori_cc} as
\begin{linenomath}
\begin{equation}\label{model:ora_cc}
\min\{b ^\top x : \mathcal{A}(x)\ge 1-\epsilon, x \in \mathcal{X} \cap \mathbb{B}^n \}.
\end{equation}
\end{linenomath}
In general,  it is hard to compute $\mathcal{A}(x)$, it involves high dimensional integrals, or in some cases, it is a black box evaluated by  simulation methods. In addition,   constraint $\mathcal{A}(x)\ge 1-\epsilon$ is highly non-convex, in general. In this section, we propose a general delayed cut generation approach to solve formulation \eqref{model:ora_cc} when an exact oracle for $\mathcal{A}(x)$ exists.


Here we address a general  approach to solve formulation \eqref{model:ora_cc} exactly. The algorithm works by solving a relaxed problem, and cutting off infeasible solutions iteratively until we find an optimal solution. Consider the generic relaxed master problem (RMP) of formulation \eqref{model:ora_cc} as
\begin{linenomath}
\begin{equation}\label{model:M_ora_cc}
\min\{b ^\top x : x \in \mathcal{C}\cap \mathcal{X} \cap \mathbb{B}^n \},
\end{equation}
\end{linenomath}
where $\mathcal{C}$ is a set of feasibility cuts added until the current iteration.  We describe a delayed constraint generation approach with the probability oracle in Algorithm \ref{alg:GDCG}. To solve formulation \eqref{model:ora_cc}, Algorithm \ref{alg:GDCG} starts with a subset of feasibility cuts in $\mathcal C$ (could be empty) in RMP \eqref{model:M_ora_cc}. At each iteration (Lines \ref{alg:1}-\ref{alg:5}), solving RMP \eqref{model:M_ora_cc}  provides an incumbent solution $\bar x$ (Line \ref{alg:2}). Then the oracle $\mathcal{A}(\bar x)$ is used as a separation routine  to check the feasibility of $\bar x$. Note that $\mathcal{A}(\bar x) \ge 1-\epsilon $ in Line \ref{alg:3} is the feasibility condition of $\bar x$.  If $\bar x$ is feasible, then we break the loop and declare the optimal solution as $\bar x$  (Lines \ref{alg:3.5} and \ref{alg:6}); otherwise, 
a subroutine  FeasibilityCut($\bar x,\kappa,\mathcal C$) is called with  input $\bar x$ and optional parameters $\kappa$.  The subroutine adds 
a feasibility cut  to the current set $\mathcal C$ (Line \ref{alg:4}) to cut off $\bar x$ in  further iterations. We  specify this subroutine and the corresponding cuts next.

\begin{algorithm}\label{alg:GDCG}
	\SetAlgoLined
	Start with an initial set of  feasibility cuts in $\mathcal C$  (could be empty)\;
	\While{$True$}
	{  \label{alg:1}
		Solve master problem \eqref{model:M_ora_cc}, and obtain an incumbent solution $\bar x$  \label{alg:2}\;
		
		\If{$\mathcal{A}(\bar x) \ge 1-\epsilon $} 
		{ 	\label{alg:3}
			\bf{break}\;
			\label{alg:3.5}}
		\Else
		{
			Call  FeasibilityCut($\bar x, \kappa,\mathcal C$) \label{alg:4}\;
		}
		\label{alg:5}		}
	Output $\bar x$ as an optimal solution. \label{alg:6}	
	\caption{An Exact Delayed Constraint Generation Algorithm with a Probability Oracle}
\end{algorithm}

Let $V_1:=\{1,\dots,n\}$. Given an incumbent solution $\bar x$ such that $\mathcal{A}(\bar x) < 1-\epsilon $, let $ J_1 = \{ i \in  V_1|\bar x_i = 1\} $ and $ J_0 = \{ j \in V_1|\bar x_j = 0 \} $. A class of feasibility cuts, commonly known as no-good cuts, is given by
\begin{linenomath}
\begin{equation}
	\label{eq:LLcut}
	\sum_{i \in J_1}(1-x_i)+\sum_{j \in J_0}x_j \ge 1,
\end{equation}   
\end{linenomath} 
which ensures that if $\bar x$ is infeasible, then at least one component in $\bar x$ must be changed.  Laporte and Louveaux \cite{Laporte1993} provide a  review of inequality \eqref{eq:LLcut}  for two-stage stochastic programs where the  first-stage problem is pure binary and  second-stage problem is mixed-integer. 


In the next proposition, we observe that if $\mathbb{P}( \mathcal{B}(x))$ is monotonically increasing in $x$ for problem \eqref{model:ori_cc}, then a stronger inequality is valid for formulation \eqref{model:ora_cc}. Throughout, we let $\mathbf {e}_{j}$ be a unit vector of dimension $n$ whose $j$th component is 1.

\begin{proposition}\label{prop:valid_LLcut_strong}
	Suppose that $\mathbb{P}( \mathcal{B}(x))$ is a monotonically increasing function in $x$.  
	Given a vector $\bar x$ with $J_0=\{i\in V_1: \bar x_i=0\}$ and $J_1=V_1\setminus J_0$ and $\mathbb P( \mathcal{B}(\bar x)) < 1-\epsilon$, let  $\kappa(J_0)<|J_0|$ be a positive integer such that $\forall \mathcal{K} \subseteq J_0$ with $|\mathcal{K}| = \kappa(J_0)-1$, we have $\mathbb P( \mathcal{B}(\bar x+\sum_{j \in \mathcal{K}  }\mathbf {e}_{j})) < 1-\epsilon$. 
	\begin{itemize}
		\item[(i)] The inequality
		\begin{equation}
			\label{eq:LLcut_strong}
			\sum_{j \in J_0}x_j \ge \kappa(J_0),
		\end{equation} 
		is valid for formulation \eqref{model:ora_cc}. 
		\item[(ii)] Inequality \eqref{eq:LLcut_strong} is stronger than inequality \eqref{eq:LLcut} for the same choice of $J_0$.
	\end{itemize}
\end{proposition}    
\begin{proof}
	\begin{itemize}
		\item[(i)] 
		Let $\bar x'\ne \bar x$ denote another vector, where $\bar J_1' = \{i \in V_1 | \bar x_i' = 1\} $ and $\bar J_1' \subset J_1$. Since $\mathbb{P}( \mathcal{B}(x))$ is a monotonically increasing function, $\bar J_1' \subset J_1$ implies that $\mathbb{P}(\mathcal{B}(\bar x')) \le \mathbb{P}(\mathcal{B}(\bar x)) < 1-\epsilon $, so $\bar x'$ is also infeasible for formulation \eqref{model:ora_cc}. Recall that for all $\mathcal{K} \subseteq J_0$, where $|\mathcal{K}| \le \kappa(J_0)-1$, we have $P( \mathcal{B}(\bar x+\sum_{j \in \mathcal{K}  }\mathbf {e}_{j})) < 1-\epsilon$. Then, for a feasible solution  $x'$, with $J_1' = \{i \in V_1| x_i' = 1 \}\revised{\supseteq J_1}$ and $\mathbb{P}( \mathcal{B}(x')) \ge 1-\epsilon$, we must have $| J_1'\setminus J_1 | \ge \kappa(J_0)$. Note that $J_1'\setminus J_1\subseteq J_0$. Hence, $\sum_{j \in J_0}x_j \ge \sum_{j \in J_1'\setminus J_1}x_j  \ge  \kappa(J_0)$, which proves the claim. 
			\item[(ii)]  Note that for the same choice of $J_0$, we have
		$\sum_{i \in J_1}(1-x_i)+\sum_{j \in J_0}x_j \ge \sum_{j \in J_0}x_j \ge \kappa(J_0) \ge 1,$
		where the first inequality follows because  $\sum_{i \in J_1}(1-x_i)\ge 0$. The result then follows. 
	\end{itemize}
\end{proof}

In light of Proposition \ref{prop:valid_LLcut_strong}, we specify the subroutine FeasibilityCut($\bar x,\kappa,\mathcal C$)   in Algorithm \ref{alg:nos_PSC} for the case that $\mathcal{A}(\bar x)$ is monotone. If  $\mathcal{A}(\bar x)$ is not monotone, then inequality \eqref{eq:LLcut_strong}  should be replaced with inequality \eqref{eq:LLcut} in Algorithm \ref{alg:nos_PSC}.  In this subroutine,  we seek inequalities \eqref{eq:LLcut_strong} with $\kappa(J_0)\le \kappa$, given the input parameter $\kappa$. If $\kappa=2$, then we check if there exists some $j\in J_0$ for which $\mathcal{A}(\bar x + \mathbf {e}_{j}) \ge 1-\epsilon$. In other words, we check if there exists a feasible solution after letting $x_j=1$ for some $j\in J_0$. If so, we let $\kappa(J_0)=1$  for the inequality to be valid.  If  such a $j$ does not exist, then this implies that complementing one variable that is in $J_0$ is not sufficient to obtain a feasible solution. In this case,  we  let $\kappa(J_0)=2$ in inequality  \eqref{eq:LLcut_strong}. Note that higher values of $\kappa$ than 2 will require more computational effort, so we only consider $\kappa=2$ \revised{in} Algorithm \ref{alg:nos_PSC}.

\begin{algorithm}[htb]\label{alg:nos_PSC}
	\SetAlgoLined
	{  \label{alg3:7}
\		\For{$j \in J_0$} 
		{\label{alg3:9}
			\If{$\mathcal{A}(\bar x + \mathbf {e}_{j}) \ge 1-\epsilon$} 
			{\label{alg3:10}
				Add inequality \eqref{eq:LLcut_strong} with $\kappa(J_0) = 1$ to $\mathcal C$\; \label{alg3:11}    
				$BoundIncrease = 0$\;
				\bf{break}\; 
			}
		}
		\If{$\kappa=2$}
		{
			Add inequality \eqref{eq:LLcut_strong} with $\kappa(J_0) = 2$ to $\mathcal C$\label{alg3:17}\;    
		}  		
	} \label{alg3:19}
	\caption{Subroutine FeasibilityCut($\bar x,\kappa,\mathcal C$)}
\end{algorithm}

Next, we demonstrate the proposed algorithm on a class of probabilistic set covering problems, and provide alternative solution approaches utilizing an efficient probability oracle.

	\section{An Application: A Probabilistic Partial Set Covering Problem}\label{sec:problem}
In this section, we study  a probabilistic partial set covering problem (PPSC) as an application of problem \eqref{model:ori_cc}.  First, we describe the deterministic  set covering problem. The deterministic set covering problem is a fundamental combinatorial optimization problem that arises in many applications, such as facility selection, scheduling, and manufacturing. We refer the reader to  \cite{Balas1983} for a review of the various applications of the set covering problem. For example, in the facility selection problem, there are  $n$ facilities given by the set $V_1$ and $m$ customers given by the set $V_2$. Suppose that  facility $j$ covers (satisfies the demand of)  customer $i$  if the travel time between the facility and the customer is within a pre-specified time limit. In this case, we can form a set $S_j$ as those customers who are within the acceptable time limit away from facility $j$.    
Given the cost of building  facility $j$, $b_j$, the set covering problem aims to find the minimum cost selection of facilities that cover all customers. 

More formally, given a set of items $V_2:=\{1,\dots,m\}$ and a collection of $n$ subsets $S_j\subseteq V_2, j\in V_1:=\{1,\dots,n\}$ such that $\cup_{j=1}^n S_j=V_2$,  the deterministic set covering problem is defined as
\begin{subequations}\label{eq:DSC}
	\begin{linenomath} 
	\begin{align}
		\min~~& \sum_{j\in V_1} b_j x_j \label{eq:DSC1}\\
		\text{s.t.}~~&\sum_{j\in V_1}  t_{ij} x_j \ge h_i & \forall i \in V_2 \label{eq:DSC2} \\
		& x\in\mathbb{B}^{n},\label{eq:DSC3}
	\end{align}
\end{linenomath} 
\end{subequations}   
where $b_j$ is the objective coefficient of $x_j$, $h_i=1$ for all $i\in V_2$, and $ t_{ij}=1$ if $i \in S_j$; otherwise, $ t_{ij} = 0$ for all $i\in V_2\setminus S_j, j\in V_1$.   Karp \cite{Karp1972} proves that the set covering problem is NP-hard.

Probabilistic set covering extends this  problem to the probabilistic setting to capture uncertain  travel times. In the stochastic variant of the set covering problem we consider, the chance constraint ensures  a high quality of service as measured by serving a target number $\tau\le m$ of the customers within preferred time limits with high probability.     
Different variants of the probabilistic set covering problem have been considered in the literature, wherein constraint  \eqref{eq:DSC2} is replaced with a chance constraint  when either the constraint coefficients $t_{ij}$ or the right-hand side $h_i$ is assumed to be random for $ i \in V_2, j \in V_1$.     
 Beraldi and Ruszczy\'{n}ski \cite{Beraldi2002b} and  Saxena et al.\ \cite{Saxena2010} study the uncertainty in the right-hand side of constraint \eqref{eq:DSC2}, in other words $h_i$ is assumed to be a binary random variable.  Fischetti and Monaci \cite{Fischetti2012} and  Ahmed and Papageorgiou \cite{Ahmed2013} study the uncertainty with the randomness in the coefficients  of constraint \eqref{eq:DSC2}, i.e., $ t_{ij}$ is a binary random variable indicating whether  set $j$ covers  item $i$. They consider {\it individual} chance constraints that ensure that each item is covered with a certain probability.    In this paper, we focus on the uncertainty in $ t_{ij}$ for all $i \in V_2$ and $j \in V_1$. In addition, we study a version of PPSC such that the probability that the selected subsets  cover a given number $\tau$ of items in $V_2$ is at least $1-\epsilon$, which we describe next. (Note that when $\tau=m$, our model considers the {\it joint} probability of covering {\it all} customers.)

Let $\sigma(x)$ be a random variable representing the number of covered items in $V_2$ for a given $x$. \revised{For example, in the facility selection problem with random travel times, $\sigma(x)$ represents the number of customers for whom the travel time from a selected facility $j\in V_1$ (with $x_j=1$) is within the pre-specified time limit.} Suppose that we are given the cost $b_i$ of each set $i \in V_1$, a target $\tau $ of the number of covered items in $V_2$, and a risk level $\epsilon\in [0,1]$. The variant of the probabilistic   set covering model we consider is
\begin{subequations}\label{eq:PPSC}
	\begin{linenomath}
	\begin{align}
		\min~~& \sum_{i\in V_1} b_i x_i \label{eq:PSC1}\\
		\text{s.t.}~~&\mathbb P( \sigma(x)\ge \tau )\ge 1-\epsilon \label{eq:PSC2} \\
		& x\in\mathbb{B}^{n},\label{eq:PSC3}
	\end{align}
	\end{linenomath}
\end{subequations}
where $\sigma(x)\ge \tau $ is the desired covering event, $\mathcal B(x)$, for a given $x$. \revised{For $\tau=m$, this problem is equivalent to a probabilistic set covering model, where each node must be covered. However, because we allow $\tau\le m$, we refer to this model as probabilistic partial set covering.} Note that $\sigma(x)$ is a submodular function \cite{KKT03}. Our goal is  to minimize the total cost of the sets selected from $V_1$ while guaranteeing a certain degree of coverage of the items in $V_2$.

We  represent the partial set covering problem on a bipartite graph $G=(V_1 \cup V_2, E)$. There are two groups of nodes $V_1$ and $V_2$ in $G$, where all arcs in $E$ are from $V_1$ to $V_2$. Node $i \in V_1$ represents set $S_i$ and  nodes in  $j \in V_2$ represent the items. There exists an arc $(i,j) \in E$ representing the covering relationship if  $j \in S_i$ for $i \in V_1$. In probabilistic   set covering, the covering relationship is stochastic, in other words  an item $i$  may not be covered by the subset $S_j$, $j\in V_1$, even though $i\in S_j$. 
In this paper, we consider probabilistic partial set covering problems (PPSC) under two probability distributions:

\noindent 	{\bf Probabilistic Partial Set Covering with  Independent Probability Coverage:}
	In this model,  each node $j$ has an independent probability $a_{ij}$ of being covered by  node $i$ for  $j\in S_i$. 
	
\noindent    {\bf Probabilistic Partial Set Covering with Linear Thresholds:}
	In the linear threshold model of  Kempe et al.\ \cite{KKT03}, each arc $(i,j)\in E$ has a deterministic weight $0\le a_{ij}\le 1$, such that for all nodes $j\in V_2$, $\sum_{i:(i,j)\in E}a_{ij}\le 1$. In addition, each node $j\in V_2$ selects a threshold $\nu_j \in [0,1]$ {\it uniformly at random}. A node $j \in V_2$ is covered if sum of the weights of its selected neighbors $i \in V_1$ is above its threshold, i.e., $\sum_{i:(i,j)\in E} a_{ij}x_i\ge \nu_j$.

The probabilistic  models we consider may be seen as chance-constrained extensions of  the independent cascade and  linear threshold models in social networks  proposed by  Kempe et al.\ \cite{KKT03},   applied to bipartite graphs.  
 Zhang et al.\ \cite{ZCAWZ14} first proposed a model to find the minimum number of individuals to influence a target number of people in social networks with a probability guarantee, where one individual has an independent probability of influencing another individual. Note that if the social network is a bipartite graph, the question proposed by Zhang et al.\ \cite{ZCAWZ14} can be formulated as PPSC.    
 Zhang et al.\ \cite{ZCAWZ14} describe a polynomial time algorithm to compute $\mathbb P( \sigma(x)\ge \tau )$ exactly for bipartite graphs under certain probability distributions for a given $x$. They  propose a greedy heuristic to obtain a solution to PPSC, which has  an $O(m+n)$ multiplicative error and an $O(\sqrt{ m+n})$ additive error on the quality of the solution for the case that $b_i=1$ for all $i\in V_1$, and no performance guarantee on the quality of the solution for the general cost case. In contrast, we give an exact algorithm to find an optimal $x$. Next, we review an efficient oracle for PPSC under the distributions of interest for both versions of PPSC.

\subsection{An Oracle} \label{sec:dp}
Let $P(x,i)$ be the probability that a given  solution $x$ covers node $i\in V_2$. For the linear threshold model, $P(x,i) = \sum_{j \in V_1} a_{j,i}x_j$, and  for the independent probability coverage model, $P(x,i) = 1- \prod_{j \in V_1} (1-a_{j,i}x_j) $. For a given $x$, the probability of covering exactly $k$  nodes in $V_2$ out of a total of $|V_2|=m$ is represented as $\mathbb P( \sigma(x) = k)$, and $\mathbb P( \sigma(x)\ge \tau ) = \sum_{k = \tau }^{m} \mathbb P( \sigma(x) = k)$. Note that $\mathbb P( \sigma(x) = k)$ is equal to the probability mass function of the Poisson binomial distribution \cite{Hoeffding1956,Samuels1965,Wang1993}, which is the discrete probability distribution of $k$ successes in $m$ Bernoulli trials, where each Bernoulli trial  has a unique success probability.   Let $\mathbb P_i( \sigma(x) = \revised{j})$ denote the probability of having \revised{$j$} covered nodes in $V_2\setminus\{i\}$  for a given $x$ \revised{for any $j=0,\ldots,k$}. 
Samuels \cite{Samuels1965} provides a formula  to obtain the value of probability mass function of the Poisson binomial distribution:
\begin{linenomath} 
\begin{equation}
	\label{eq:recursive}
	\mathbb P( \sigma(x) = \revised{j}) =  P(x,i) \times \mathbb P_i( \sigma(x) =  \revised{j}-1) + (1-P(x,i)) \times \mathbb P_i( \sigma(x) = \revised{j}).
\end{equation}
\end{linenomath} 
Next, we describe a dynamic program (DP) to compute $\mathbb P( \sigma(x)\ge \tau )$ exactly \cite{Barlow1984,ZCAWZ14}.  Let $V^i = \{1,\dots, i\} \in V_2$ be the set of the first $i$ nodes of $V_2$. Also let $A(x,i,j)$ represent the probability that the selection $x$ covers $j$ nodes among $V^i$ for $0\le j\le i, i\in V_2$.  The DP recursion for  $A(x,i,j)$ for $1\le j\le i, i\in V_2$ is formulated as
\ignore{
$$ A(x,1,j)=\left\{
\begin{aligned}
&P(x,1), &  j=1 \\
&1-P(x,1), & j=0 \\
\end{aligned}
\right.
$$ 
and for $i>1$, 
}
\begin{linenomath} 
$$ A(x,i,j)=\left\{
\begin{aligned}
&A(x,i-1,j) \times (1-P(x,i)), &  j=0 \\
&A(x,i-1,j) \times (1-P(x,i)) + A(x,i-1,j-1) \times P(x,i), & 0<j<i  \\
&A(x,i-1,j-1)\times P(x,i), & j=i, \\
\end{aligned}
\right.
$$
\end{linenomath} 
where the boundary condition is $A(x,0,0) = 1$. The goal function $\sum_{j= \tau }^{m} A(x,m,j)$ calculates the probability that the number of covered nodes is at least the target $\tau $ for a given $x$. In other words, $\mathcal{A}(x) =\mathbb P( \sigma(x)\ge \tau ) =   \sum_{j= \tau }^{m} A(x,m,j)$. For a given $x$, the running time of this DP is $\mathcal{O}(nm+m^2)$, because obtaining $P(x,j)$ for all $j \in V_2$ is  $\mathcal{O}(nm)$, and computing the recursion is $\mathcal{O}(m^2)$. 

Next, we show that $ \mathcal{A}(x)$ is a monotone increasing function in $x$. We use the following lemma as a tool to prove the property of $\mathbb P( \sigma(x)\ge \tau )$. 

\begin{lemma}\label{lemma:coupling} [Lemma 2.12 in \cite{Hofstad2016}]
	Let $R_1$ and $R_2$ be two random variables defined on two different probability spaces. The random variables $\hat R_1$ and $\hat R_2$ are a coupling of $R_1$ and $R_2$ when $\hat R_1$ and $\hat R_2$ are defined in the same probability space. In addition, the marginal distribution of $\hat R_1$ is the same as $R_1$, and the marginal distribution of $\hat R_2$ is the same as $R_2$. Given $\gamma \in \mathbb{R}$, $\mathbb{P}(R_1 \ge \gamma) \le \mathbb{P}(R_2 \ge \gamma)$ if and only if there exists a coupling $\hat R_1$ and $\hat R_2$ of $R_1$ and $R_2$ such that $\mathbb{P}(\hat R_1 \le \hat R_2) = 1$.
\end{lemma}

\begin{proposition}\label{prop:increasingP}
	Given $\tau $, $\mathbb P( \sigma(x)\ge \tau )$ is a monotonically increasing function in $x$ for the probabilistic partial set covering problem under the independent probability coverage  and the linear threshold models.
\end{proposition}
\begin{proof}
	Suppose that we have two binary vectors, $x'$ and $x''$. Let $X'$ and $X''$ be the support of the vectors $x'$ and $x''$, and suppose that $X' \subseteq X''$. Recall that for the linear threshold model, $P(x,i) = \sum_{j \in V_1} a_{j,i}\revised{x_j}$, and  for the independent probability coverage model, $P(x,i) = 1- \prod_{j \in V_1} (1-a_{j,i}x_j) $. Since $P(x,i)$ is a monotone increasing function in $x$, $P(x',i) \le P(x'',i)$ for all $i = 1,\dots,m$. We construct two random variables $\hat \sigma(x')$ and $\hat \sigma(x'')$, which is a coupling of $\sigma(x')$ and $\sigma(x'')$. We apply the technique shown in \cite{Pollard2001} (Chapter 10, Example 1) to generate $\hat \sigma(x')$ and $\hat \sigma(x'')$. Let $U_i$ be an independent random variable with uniform distribution in $[0,1]$ for all $i=1,\dots,m$. Let $\hat \sigma_i(x')$ be a random variable for all $i = 1,\dots,m$ such that
$ \hat \sigma_i(x')= 1 $ if $ U_i \le P(x',i) $, and 0 otherwise. Let $\hat \sigma_i(x'')$ be a random variable for all $i = 1,\dots,m$  such that
	$ \hat \sigma_i(x'')=1$ if $  U_i \le P(x'',i)$, and 0 otherwise.
	We set $\hat \sigma(x') = \sum_{i=1}^{m} \sigma_i(x')$ and $\hat \sigma(x'') = \sum_{i=1}^{m} \sigma_i(x'')$. Since $P(x',i) \le P(x'',i)$, $\hat \sigma_i(x') \le \hat \sigma_i(x'')$ for all $i = 1,\dots,m$. Then, $\hat \sigma(x') \le \hat \sigma(x'')$. From Lemma \ref{lemma:coupling}, if $\mathbb P(\hat \sigma(x') \le \hat \sigma(x'')) = 1$, then $\mathbb P( \sigma(x')\ge \tau ) \le \mathbb P( \sigma(x'')\ge \tau )$. This completes the proof.   	   	
\end{proof}

\ignore{
	$$ \hat \sigma_i(x')=\left\{
	\begin{aligned}
	&1, &  U_i \le P(x',i) \\
	&0, & otherwise. \\
	\end{aligned}
	\right.
	$$     	
	Let $\hat \sigma_i(x'')$ be a random variable for all $i = 1,\dots,m$  such that
	$$ \hat \sigma_i(x'')=\left\{
	\begin{aligned}
	&1, &  U_i \le P(x'',i) \\
	&0, & otherwise. \\
	\end{aligned}
	\right.
	$$     	
	}

Proposition \ref{prop:increasingP} proves the intuitive result that if more nodes from $V_1$ are selected, then we have a higher chance to cover more nodes from $V_2$. Consequently, Proposition \ref{prop:increasingP} allows us to use the stronger inequality \eqref{eq:LLcut_strong} in Algorithm \ref{alg:GDCG}. In the following subsections, we employ the DP described in this section first to reformulate the problem using a compact mathematical model, and then in Algorithm \ref{alg:nos_PSC} as the oracle $\mathcal{A}(x)$ for  PPSC.

\subsection{A Compact MIP for PPSC with a Probability Oracle} 
Using the DP representation of the oracle $\mathcal{A}(x)=\mathbb{P}( \sigma(x)\ge \tau )$,  PPSC problem \eqref{eq:PPSC} can be reformulated as a compact mathematical program
\begin{subequations}\label{eq:PPSC-NLP}
	\begin{linenomath} 
	\begin{align}
		\min~~& \sum_{i\in V_1} b_i x_i \label{eq:PPSC-NLP-obj}\\
		\text{s.t.}~~    	& \bar A_{0,0} = 1\label{eq:PPSC-bdry}\\
		&  \bar A_{i,j} =  \bar A_{i-1,j}   (1-P(x,i)),\quad   i=1,\dots, m; j=0\label{eq:PPSC-NLP-dp1}\\
		&  \bar A_{i,j} =  \bar A_{i-1,j}   (1-P(x,i)) +  \bar A_{i-1,j-1}   P(x,i), \\
		& \quad i=1,\dots, m; 0<j<i \notag\\
		&  \bar A_{i,j} =   \bar A_{i-1,j-1}  P(x,i) ,\quad i=1,\dots, m; j=i\label{eq:PPSC-NLP-dp3}\\
		& \sum_{j= \tau }^{m}  \bar A_{m,j} \geq 1-\epsilon \label{eq:PPSC-NLP-goal}\\
		& x\in\mathbb{B}^{n} \\
		&  \bar A_{i,j} \in \mathbb{R_+},\quad 0\le j \le i \le  m, \label{eq:g_final}
	\end{align}
\end{linenomath} 
\end{subequations}
where $ \bar A_{i,j}$ is a decision variable representing $A(x,i,j)$ defined in the DP formulation (we drop the dependence on $x$ for ease of notation).  Constraint \eqref{eq:PPSC-bdry} is the boundary condition of the DP, constraints \eqref{eq:PPSC-NLP-dp1}--\eqref{eq:PPSC-NLP-dp3} are the DP recursive functions, and constraint \eqref{eq:PPSC-NLP-goal} is the goal function. Note that $P(x,i)$ is a function of the decision vector $x$. Hence, formulation \eqref{eq:PPSC-NLP} is  a mixed-integer {\it nonlinear} program due to the  
constraints \eqref{eq:PPSC-NLP-dp1}-\eqref{eq:PPSC-NLP-dp3}.     Depending on the complexity of the function $P(x,i)$, this formulation may be difficult to solve. 
However, for a special case of PPSC, namely the linear threshold model, the nonlinear programming model can be reformulated as a {\it linear} mixed-integer program (MIP). Recall that for the linear threshold model, $P(x,i) = \sum_{u \in V_1} a_{u,i}x_u$. 
\ignore{
Therefore, we can represent $A(i,j)$ as $ \bar A_{i,j}$, where 
the boundary condition is $\bar A(0,0) = 1$,  and    the recursive formulation of $ \bar A_{i,j}$ for $1\le j\le m$ is 
\ignore{
$$\bar A(1,j)=\left\{
\begin{aligned}
& \sum_{j \in V_1} a_{j,1}x_j, &  j=1 \\
&1-\sum_{j \in V_1} a_{j,1}x_j, & j=0 \\
\end{aligned}
\right.
$$ and
}
$$ \bar A_{i,j}=\left\{
\begin{aligned}
& \bar A_{i-1,j}   (1-\sum_{u \in V_1} a_{u,i}x_u), &  j=0 \\
& \bar A_{i-1,j}   (1-\sum_{u \in V_1} a_{u,i}x_u) +  \bar A_{i-1,j-1}   \sum_{u \in V_1} a_{u,i}x_u, & 0<j<i  \\
& \bar A_{i-1,j-1}  \sum_{u \in V_1} a_{u,i}x_u, & j=i. \\
\end{aligned}
\right.
$$
The recursions have two kinds of variables $ \bar A_{i,j}$ and $x_j, j\in V_1$. 
}
Hence, the term $ \bar A_{i,j}   x_j$ appearing in \eqref{eq:PPSC-NLP-dp1}-\eqref{eq:PPSC-NLP-dp3} is a bilinear term, which can be linearized \cite{McC76,Adams1993}.  
To this end, we introduce the additional variables $\gamma_{u,i,j} =  \bar A_{i,j}   x_u$ for $u\in V_1, i\in V_2, 0\le j\le i$,  and obtain  an equivalent linear MIP \begin{subequations}\label{eq:PPSC-LTMIP}
	\begin{linenomath} 
	\begin{align}
		\min~~& \sum_{i\in V_1} b_i x_i \label{eq:PPSC-LTMIP-obj}\\
		\text{s.t.}~~ 	& \eqref{eq:PPSC-bdry}, \eqref{eq:PPSC-NLP-goal}-\eqref{eq:g_final}\\
		& \bar A_{i,j} = \bar A_{i-1,j}-\sum_{u \in V_1}a_{u,i} \gamma_{u,i-1,j},\quad i=1,\dots, m; j=0\label{eq:LTMIP-dp1}\\
		& \bar A_{i,j} = \bar A_{i-1,j}-\sum_{u \in V_1}a_{u,i} \gamma_{u,i-1,j} + \sum_{u \in V_1}a_{u,i} \gamma_{u,i-1,j-1},\label{eq:LTMIP-dp2}\\
		& \quad i=1,\dots, m; 0<j<i \notag\\
		& \bar A_{i,j} =  \sum_{u \in V_1}a_{u,i} \gamma_{u,i-1,j-1}, \quad i=1,\dots, m; j=i\label{eq:LTMIP-dp3}\\
		& \gamma_{u,i,j} \leq x_u,\quad i=0,\dots, m;j=0,\dots, i; u \in V_1 \label{eq:LTMIP-gamabdy1}\\
		& \gamma_{u,i,j} \leq  \bar A_{i,j},\quad i=0,\dots, m;j=0,\dots, i; u \in V_1 \\
		& \gamma_{u,i,j} \geq  \bar A_{i,j}-(1-x_u),\quad i=0,\dots, m;j=0,\dots, i; u \in V_1 \label{eq:LTMIP-gamabdy3}\\
		& \gamma_{u,i,j} \geq 0,\quad i=0,\dots, m;j=0,\dots, i; u \in V_1 . \label{eq:LTMIP-last}
	\end{align}
\end{linenomath} 
\end{subequations}
The DP recursion is represented in constraints \eqref{eq:LTMIP-dp1}-\eqref{eq:LTMIP-dp3}. Constraints \eqref{eq:LTMIP-gamabdy1}-\eqref{eq:LTMIP-last} are the McCormick linearization constraints to ensure that 
if $x_u = 0$, then $\gamma_{u,i,j}=0$, and if  $x_u = 1$, then $\gamma_{u,i,j} =  \bar A_{i,j}$.     
As a result, in this special case, the oracle is a formulable function, and the linear threshold model can be solved exactly with the compact MIP \eqref{eq:PPSC-LTMIP}. Alternatively, Algorithm \ref{alg:GDCG} can be used to solve the exponential formulation of \eqref{eq:PPSC} by delayed constraint generation. We close this subsection by noting that, for the case of the independent probability coverage model, the  mixed-integer nonlinear program \eqref{eq:PPSC-NLP} is multilinear due to the terms $   \bar A_{i-1,j-1}  (1- \prod_{u \in V_1} (1-a_{u,i}x_u))$. While such multilinear terms can also be linearized with successive application of the McCormick linearization, the resulting formulations are large scale and they suffer from weak LP relaxations. Therefore, we do not pursue such formulations for the independent probability coverage model in our computational study. 

\subsection{A General Decomposition Approach for PPSC with a Probability Oracle} 
Algorithm \ref{alg:GDCG} can be used to solve formulation \eqref{eq:PPSC} exactly. To update $\kappa(J_0)$ in Algorithm \ref{alg:nos_PSC} for $\kappa=2$ more efficiently, we utilize the DP structure of $\mathcal{A}(\bar x)$. Note that when we obtain a solution $\bar x$, with an associated $J_0$, we first calculate $\mathcal{A}(\bar x)$ using the DP, which requires the calculation of $P(\bar x,i), \forall i \in V_2$. Then, to calculate inequalities with $\kappa(J_0)=2$, we need to calculate $\mathcal{A}(\bar x+\mathbf e_j)$ for each $j \in J_0$. If we calculate $P(\bar x+\mathbf e_j,i), \forall i \in V_2$ for each $j \in J_0$ from scratch, the
 time complexity  is $\mathcal{O}(nm)$ for the independent probability coverage  and the linear threshold models. However, given that we have just calculated (and stored) all $P(\bar x,i)$ values, the time complexity of updating $P(\bar x,i)$ to $P(\bar x + \mathbf {e}_{j} ,i), \forall i \in V_2, j\in J_0$  is $\mathcal{O}(m)$.

As we will show in our computational study, when the number of decision variables is large,  Algorithm \ref{alg:GDCG} exhibits slow convergence, even if there exists an efficient probability oracle. In this case, Algorithm \ref{alg:GDCG} may check a large (worst case exponential) number of incumbent solutions $\bar x$ to obtain the optimal solution. In the next section, we consider a sampling-based approach to find approximate solutions to PPSC. The sampling-based approach enables us to use the problem structure to expedite the convergence to a solution,  but the optimal solution to the sample approximation problem may not be feasible with respect to the true distribution. In this case,  the probability oracle is used as a detector to check and correct the infeasibility of the solution given by the sampling-based approach.	

	\subsection{A Sampling-Based Approach for PPSC with a Probability Oracle} 

Using sampling-based methods, we can approximately represent the uncertainty with a finite number of possible outcomes (known as scenarios). This, in turn, allows us to rewrite the
non-convex chance constraint as linear inequalities with big-M coefficients, if the desirable event $\mathcal B(x)$ has a linear representation. Such a formulation is known as the deterministic equivalent formulation.  Luedtke et al.; K\"{u}\c{c}\"{u}kyavuz; Abdi and Fukasawa; Zhao et al.\ and  Liu et al.\ \cite{Luedtke2010,Simge2012,AF16,Zhao2016,LKL17} introduce strong valid inequalities for the deterministic equivalent formulation of linear chance constraints under  right-hand side    uncertainty. Ruszczy\'{n}ski; Beraldi and Bruni;  Lejeune;  Luedtke and  Liu et al.\ \cite{Ruszczynski2002,Beraldi2010,L12,Luedtke2014,LKL16} study  general CCPs with the randomness in the   coefficient (technology) matrix  (including two-stage CCPs), and propose solution methods for the sampling-based approach. In another line of work,  Song et al.\ \cite{Song2014} consider a special case  of combinatorial chance-constrained programs, namely the chance-constrained packing problems under finite discrete distributions,  and give a 
delayed constraint generation algorithm using the so-called probabilistic cover and pack inequalities valid for the chance-constrained binary packing problems.

In this section, we consider related sampling-based reformulations of PPSC. 
First, we describe how we sample from the true distribution to 
obtain a set of scenarios (sample paths) $\Omega$ for PPSC (see \cite{KKT03,first2016} for a  detailed description). For the case of the independent probability coverage model, we 
generate a scenario  by tossing biased coins for each arc $(i,j) \in E$ with associated  probability $a_{ij}$. The coin tosses reveal if   node $j\in V_2$ is covered by  node $i\in V_2$ in which case we refer to arc $(i,j)\in E$ as a live arc.   For each sample (scenario) $\omega \in \Omega$, with a probability of occurrence  $p_\omega$, a so-called {\it live-arc graph} $G_\omega=(V_1 \cup V_2,E_\omega)$ is constructed, where $E_\omega$ is the set of live arcs  under scenario $\omega$.  We refer the reader to  Kempe et al.\ \cite{KKT03} for a scenario generation method for the linear threshold model, which results in live-arc graphs $G_\omega$ for each $\omega\in \Omega$. It is important to note that in  the  live-arc graphs of linear threshold models, each node in $ V_2$  has at most one incoming arc. Let $t_{ij}^\omega=1$ if arc $(i,j)\in E_\omega$ for $\omega\in \Omega$, and $t_{ij}^\omega=0$ otherwise.

Given a set of scenarios $\Omega$ and a target level  $\tau$, 
we can reformulate the submodular formulation \eqref{eq:PPSC} of PPSC as a two-stage  chance-constrained program under a finite discrete distribution. In the first stage, the nodes from set $V_1$ are selected by the decision vector  $x$. Then the uncertainty unfolds, and     live arcs are realized. The second-stage problem for each scenario determines  the number of nodes in $V_2$   covered by the nodes in $V_1$ selected in the first stage.  Let $y_i^\omega=1$ if node $i\in V_2$ is covered by the node selection $x$ under scenario $\omega\in \Omega$. 
Then a deterministic equivalent formulation for PPSC is
\begin{subequations}\label{eq:sample-DEP}
	\begin{linenomath} 
	\begin{align}
	\min~~& \sum_{j\in V_1} b_j x_j \label{eq:sample-DEP1}\\
	\text{s.t.}~~&\sum_{j\in V_1}  t_{ij}^{\omega} x_j \ge y_i^{\omega} & \forall i \in V_2 , \forall \omega \in \Omega \label{eq:sample-DEP2} \\
	& \sum_{i \in V_2} y_i^{\omega} \ge \tau z_\omega & \forall \omega \in \Omega \label{eq:sample-DEP3}\\
	&\sum_{\omega \in \Omega} p_{\omega} z_{\omega}\ge 1-\epsilon  \label{eq:sample-DEP4}\\
	& x\in \mathbb{B}^{n},y\in \mathbb{B}^{m \times |\Omega|},z\in \mathbb{B}^{|\Omega|}, \label{eq:sample-DEP5}
	\end{align}
\end{linenomath} 
\end{subequations} where constraints  \eqref{eq:sample-DEP2} ensure that $y_i^\omega=1$ if node $i\in V_2$ is covered by the node selection $x$ under scenario $\omega\in \Omega$, and  constraints  \eqref{eq:sample-DEP3} ensure that if $z_\omega = 1$, then  $\sum_{i \in V_2} y_i^{\omega} \ge \tau$ for all $\omega\in \Omega$. Constraint \eqref{eq:sample-DEP4} ensures that the probability that $\tau$ items are covered in $V_2$ is at least $1-\epsilon$.  Formulation \eqref{eq:sample-DEP}  is a very large-scale MIP that continues to challenge the state-of-the-art optimization solvers. Instead, delayed constraint generation methods akin to Benders decomposition method \cite{Luedtke2014,LKL16} are known to be computationally more effective for such problems. Because the second-stage problem is concerned with feasibility only, the decomposition algorithm proposed in \cite{Luedtke2014} is applicable to this formulation (Liu et al.\  \cite{LKL16} also consider the second-stage objective). However, we observe that \revised{the deterministic set covering problem, which is known to be NP-hard, can be reduced to one of the subproblems required for this algorithm.} Our computational results show that the  need to solve a large number of \revised{such}  difficult integer programming subproblems makes this algorithm prohibitive for the PPSC application. In this paper, we propose an alternative approach and use the submodularity property of PPSC to solve formulation \eqref{eq:sample-DEP}, which we describe next.

For each scenario $\omega \in \Omega$, let $\sigma_{\omega} (x)$ denote the number of nodes  in $V_2$ covered by the selection $x$ in the live-arc graph $G_\omega=(V_1 \cup V_2,E_\omega)$. It is known  that $\sigma_{\omega} (x)$ is submodular  \cite{VH93,KKT03}.  Given a set of scenarios $\Omega$ and target level $\tau $, we formulate PPSC as 
\begin{subequations}\label{eq:PPSC-sub}
	\begin{linenomath} 
	\begin{align}
	\min~~& \sum_{i \in V_1} b_i x_i \\
	\text{s.t.}~~ & \sigma_{\omega} (x) \ge \tau z_{\omega} & {\omega}\in \Omega \label{eq:PPSC-sub-1}\\
	&\sum_{\omega \in \Omega} p_{\omega} z_{\omega}\ge 1-\epsilon  \label{eq:PPSC-sub-2}\\
	& x\in \mathbb{B}^n, z\in \mathbb{B}^{|\Omega|}, 
	\end{align}
\end{linenomath} 
\end{subequations} where $z_\omega = 1$ implies that for a given $x$, $\sigma_{\omega} (x) \ge \tau$ is enforced. Constraint \eqref{eq:PPSC-sub-2} ensures that the probability that $\sigma_{\omega} (x) \ge \tau$ is at least $1-\epsilon$. Constraint \eqref{eq:PPSC-sub-1} involves a submodular function. To reformulate it using  linear inequalities, we introduce additional variables $\theta_{\omega}$ that represent the number of covered nodes in $V_2$ under scenario $\omega\in \Omega$. In what follows, we use the notation  $\sigma(x)$ for a given $x\in \mathbb B^n$ and  $\sigma(X)$ for the corresponding support $ X\subseteq V_1$ interchangeably, and the usage will be clear from the context. For a given  $\omega \in \Omega$, consider the polyhedron
$\mathcal S_\omega=\{(\theta_\omega,x)\in \mathbb{R} \times\{0,1\}^n:\theta_\omega\le \sigma_\omega(S) +\sum_{j\in V_1\setminus S} \rho^\omega_j(S) x_j, \forall S\subseteq V_1\}$,
 where $\rho_j^\omega(S)=\sigma_\omega(S\cup\{j\})-\sigma_\omega(S)$ is the marginal contribution of adding $j\in V_1\setminus S$ to the set $S$.  Nemhauser and Wolsey  \cite{NW81} show that when $\sigma_\omega(x)$ is nondecreasing and submodular $\max_x \sigma_\omega(x)$ is equivalent to $\max_{\theta_\omega,x} \{\theta_\omega: (\theta_\omega,x)\in \mathcal S_\omega\}.$

Note that we may need an exponential number of inequalities to represent the submodular function using linear inequalities. Instead of adding these inequalities a priori, we follow a delayed cut generation approach  that combines  Benders decomposition with the probability oracle to solve PPSC. The corresponding relaxed RMP is defined as 
\begin{subequations}\label{eq:PPSC-sub-Master}
	\begin{linenomath} 
	\begin{align}
	\min~~& \sum_{i \in V_1} b_i x_i \label{eq:PPSC-sub-Master-1}\\
	\text{s.t.}~~& (\theta_\omega,x) \in \mathcal{\bar C} \label{eq:PPSC-sub-Master-2}\\
	&\theta_{\omega}\ge \tau z_{\omega} & {\omega}\in \Omega \label{eq:PPSC-sub-Master-3}\\\
	&\sum_{\omega \in \Omega} p_{\omega} z_{\omega}\ge 1-\epsilon  \\
	& x\in \mathbb{B}^n, z\in \mathbb{B}^{|\Omega|}, \theta\in \mathbb R_+^{|\Omega|}, \label{eq:PPSC-sub-Master-4}
	\end{align}
\end{linenomath} 
\end{subequations}
where $\mathcal{\bar C}$ is the set of feasibility cuts associated with the decision variables $(\theta_\omega,x)$ for $\omega \in \Omega$. 	
In particular, given incumbent solution, $\bar x$,  of RMP  \eqref{eq:PPSC-sub-Master}, and its corresponding support \revised{$\bar X= \{i \in V_1:\bar{x}_i = 1\}$}, a submodular feasibility cut  \cite{NW81,NW88} is 
\begin{linenomath} 
\begin{equation}\label{eq:sub_cut}
\theta_\omega\le \sigma_\omega(\bar X) +\sum_{j\in V_1\setminus \bar X} \rho^\omega_j(\bar X) x_j.
\end{equation}
\end{linenomath} 

\ignore{
For a given incumbent solution, $\bar x$, which is a characteristic vector of the set $\bar X$, and scenario $\omega\in \Omega$, we define $\sigma_\omega(\bar X)$ as the number of covered nodes in $V_2$ for a given set of nodes $\bar X\subseteq V_1$ and $\rho_j^\omega(\bar X)=\sigma_\omega(\bar X \cup\{j\})-\sigma_\omega(\bar X)$ is the marginal contribution of adding $j\in V_1 \setminus \bar X$ to the set $\bar X$. Then the submodular feasibility cut 
\begin{linenomath} 
\begin{equation}\label{eq:sub_cut}
\theta_\omega\le \sigma_\omega(\bar X) +\sum_{j\in V_1\setminus \bar X} \rho^\omega_j(\bar X) x_j,
\end{equation}
\end{linenomath} 
is valid,  
because the submodular function $\sigma_\omega(x)$ is nondecreasing  \cite{NW81}. 
}

At each iteration of the algorithm, we solve RMP \eqref{eq:PPSC-sub-Master} to obtain an incumbent solution $(\bar x,\bar \theta,\bar z)$, which is used to generate the  submodular cuts \eqref{eq:sub_cut}, if necessary. 
In a related study,	Wu and K\"{u}\c{c}\"{u}kyavuz \cite{first2016} apply inequality \eqref{eq:sub_cut} to solve the stochastic  influence maximization problem, which aims to find a subset of $k$ nodes to reach the maximum expected number of nodes in a general (non-bipartite) network. Wu and K\"{u}\c{c}\"{u}kyavuz \cite{first2016} give   conditions under which inequalities \eqref{eq:sub_cut} are facet defining for  $\mathcal S_\omega$. We extend the work of \cite{first2016}, and propose a new class of  valid inequalities for the bipartite case. Before we give our proposed inequality, we provide a useful definition.

\begin{definition}
Given a live-arc graph $G_\omega=(V_1\cup V_2,E_\omega)$, if there exists an arc $(i,j) \in E_\omega$, where $i \in V_1$ and $j \in V_2$, then we say that $j$ is reachable from $i$. Given a set of nodes $B \subseteq V_1$, if  node $j \in V_2$ is reachable from all nodes in $B$ and $|B| \ge 2$, we say that $j$ is a common node of all nodes in $B$. Given two sets of nodes $B \subseteq V_1$ and $N \subseteq V_1$, where $B \cap N = \emptyset$, we define 		
$	\mathcal{U}_\omega (B, N) = \{j \in V_2: (i,j) \in E_\omega, \ \forall i \in B; \ (i,j) \notin E_\omega, \ \forall i \in  N \}$ as the set of nodes reachable from all nodes in $B$ but not reachable from any node in $ N$. For $k\in V_1$, let $\eta_\omega^k= |\mathcal{U}_\omega(\{k\}, V_1 \setminus \{k\})|$.
\end{definition}

Next we give a new class of valid inequalities.

\begin{proposition}\label{prop:valid_cineq}
	Given $D  \subseteq V_1$, and sets  $C_1^k \subseteq V_1$, and $C_2^k \subseteq V_2$ for $k = 1,\dots,c$ for some $c\in \mathbb Z_+$ such that $|C_1^k| \ge 2$, each pair of distinct nodes $\{i,j\} \in C_1^k$ satisfies $|\mathcal{U}_\omega (\{i,j\}, V_1 \setminus C_1^k) \cap C_2^k| = :n_\omega(C_1^k)$ for some $n_\omega(C_1^k)\in \mathbb Z_+$, and $C_2^i \cap C_2^j = \emptyset$ for all $i = 1,\dots,c $ and $j = 1,\dots,c $ with $i \neq j$, the inequality
	\begin{linenomath} 
	\begin{equation}\label{eq:CIneq}
	\theta_\omega \le \sum_{k = 1}^{c } n_\omega(C_1^k) \Big(1 -\sum_{j \in C_1^k} x_j\Big) + \sum_{k \in D } \eta_\omega^k (1 - x_k)	+ \sum_{j \in V_1} \sigma_\omega(\{j\})x_j
	\end{equation} 
\end{linenomath} 
	is valid for  $\mathcal S_\omega$.
	
	
\end{proposition}

\begin{proof}
	Consider a feasible point $( \hat \theta_\omega,  \hat  x)\in \mathcal S_\omega$. Note that we must have $ \hat  \theta_\omega\le \sigma_\omega( \hat X)$ at a feasible point, where  $ \hat X = \{i \in V_1:{ \hat x_i} = 1\}$. 
		Let $C''\subset \{1,\dots,c\}$, where  $\sum_{j \in C_1^k} x_j> 1$ for each $k \in C''$. Let $C' =\{1,\dots,c\} \setminus C''$, where  $\sum_{j \in C_1^k} x_j\le 1$ for each $k \in C'$. We create an additional dummy node $d$ in $V_1$, where $(d,v) \notin E_\omega$ for all $v \in V_2$, $\sigma_\omega( \{d\}) = 0$ and $\sigma_\omega(\hat X) = \sigma_\omega(\hat X \cup \{d\})$. For each $v \in V_2$, we define $r(v):=\min\{i\in \hat X:(i,v)\in E_w\}, $ if there exists $ (i,v) \in E_\omega$  for some $i \in \hat X$, we let $r(v):=d$, otherwise. 
		\ignore{
		$$ r(v):=
		\begin{cases}
		\min\{i\in \hat X:(i,v)\in E_w\}, & \text{ if }\exists (i,v) \in E_\omega \text{ for some }i \in \hat X,   \\
		d, & \text{ otherwise}, \\
		\end{cases}
		$$ 
		}
		Here $r(v) \neq d$ denotes the node that belongs to $\hat X$ and can reach $v \in V_2$. Let $R:=\cup_{v \in V_2}\{r(v)\}$. Recall the condition that $C_2^i \cap C_2^j = \emptyset$ for all $i,j = 1,\dots,c $ with $i \neq j$.  In other words, each $v \in V_2$ belongs to at most one $C_2^k$ for all $k=1,\dots,c$. For each $k \in C''$, since $\sum_{j \in C_1^k} x_j> 1$ and each pair of distinct nodes $\{i,j\} \in C_1^k$ satisfies $|\mathcal{U}_\omega (\{i,j\}, V_1 \setminus C_1^k) \cap C_2^k| = n_\omega(C_1^k)$ for some $n_\omega(C_1^k)\in \mathbb Z_+$, there  exists $v \in C_2^k$ such that $r(v) \in C_1^k$.  Thus, for each $k \in C''$, we define
		$r_k := \min\{r(v) \in R\cap C_1^k : v \in C_2^k \},$ 
		where $r_k$ denotes the node that belongs to $\hat X \cap C_1^k$ and can reach some node in $C_2^k$. From the previous discussion, $r_k$ exists for all $k\in C''$. Because  $ \hat  \theta_\omega\le \sigma_\omega( \hat X)$ at a feasible point, we have 
	\begin{linenomath} 
			\begin{align}
			\hat \theta_\omega &\le \sigma_\omega( \hat X) \notag\\
			&= \sum_{j \in \hat X} \sigma_\omega(\{j\})  - \sum_{v \in V_2} \sum_{j \in \hat X \setminus \{ r(v) \} } |v \cap \mathcal{U}_\omega (\{j,r(v)\},\emptyset)| \label{eq:valid_1}\\
			&= \sum_{j \in V_1} \sigma_\omega(\{j\}) \hat x_j - \sum_{v \in V_2} \sum_{j \in V_1 \setminus \{ r(v) \} } |v \cap \mathcal{U}_\omega (\{j,r(v)\},\emptyset)| \hat x_j \label{eq:valid_2}\\
			&\le \sum_{j \in V_1} \sigma_\omega(\{j\}) \hat x_j - \sum_{k\in C''}\sum_{v \in C_2^k} \sum_{j \in V_1 \setminus \{ r(v) \} } |v \cap \mathcal{U}_\omega (\{j,r(v)\},\emptyset)| \hat x_j \label{eq:valid_3}\\
			&\le \sum_{j \in V_1} \sigma_\omega(\{j\}) \hat x_j - \sum_{k\in C''}\sum_{v \in C_2^k} \sum_{j \in C_1^k \cap \{ V_1 \setminus \{ r(v) \} \} } |v \cap \mathcal{U}_\omega (\{j,r(v)\},\emptyset)| \hat x_j \label{eq:valid_4}\\
			&\le \sum_{j \in V_1} \sigma_\omega(\{j\}) \hat x_j -\sum_{k\in C''}\sum_{v \in C_2^k} \sum_{j \in C_1^k \cap \{ V_1 \setminus \{  r_k \} \} } |v \cap \mathcal{U}_\omega (\{j, r_k\},\emptyset)| \hat x_j \label{eq:valid_5}\\
			&= \sum_{j \in V_1} \sigma_\omega(\{j\}) \hat x_j - \sum_{k\in C''} \sum_{j \in C_1^k \cap \{ V_1 \setminus \{  r_k \} \} } | C_2^k \cap \mathcal{U}_\omega (\{j, r_k\},\emptyset)| \hat x_j \label{eq:valid_6}\\
			&\le \sum_{j \in V_1} \sigma_\omega(\{j\}) \hat x_j - \sum_{k\in C''} \sum_{j \in C_1^k \cap \{ V_1 \setminus \{  r_k \} \} } | C_2^k \cap \mathcal{U}_\omega (\{j, r_k\},V_1 \setminus C_1^k)| \hat x_j \label{eq:valid_7}\\
			&= \sum_{j \in V_1} \sigma_\omega(\{j\}) \hat x_j - \sum_{k\in C''} \sum_{j \in C_1^k \cap \{ V_1 \setminus \{  r_k \} \} } n_\omega(C_1^k) \hat x_j \label{eq:valid_8}\\
			&= \sum_{j \in V_1} \sigma_\omega(\{j\}) \hat x_j - \sum_{k\in C''} \sum_{j \in C_1^k \cap \{ V_1 \setminus \{  r_k \} \} } n_\omega(C_1^k) \hat x_j  + \sum_{k\in C''} n_\omega(C_1^k)(1-\hat x_{r_k}) \label{eq:valid_9}\\
			&= \sum_{j \in V_1} \sigma_\omega(\{j\}) \hat x_j - \sum_{k\in C''} \sum_{j \in C_1^k \cap V_1  } n_\omega(C_1^k) \hat x_j  +  \sum_{k\in C''} n_\omega(C_1^k) \label{eq:valid_10}\\
			&\le \sum_{j \in V_1} \sigma_\omega(\{j\}) \hat x_j + \sum_{k\in C''} n_\omega(C_1^k) \Big(1- \sum_{j \in C_1^k}\hat x_j \Big) + \sum_{k\in C'} n_\omega(C_1^k) \Big(1 -\sum_{j \in C_1^k}  \hat x_j\Big) \label{eq:valid_11}\\
			&\le  \sum_{j \in V_1} \sigma_\omega(\{j\}) +  \sum_{k = 1}^{c } n_\omega(C_1^k) \Big(1 -\sum_{j \in C_1^k}  \hat x_j\Big) 	+ \sum_{k \in D } \eta_\omega^k (1 -  \hat x_k). \label{eq:valid_12}
			\end{align}
		\end{linenomath} 
		Equality \eqref{eq:valid_1} follows from the definition of $\sigma_\omega( \hat X)$ for a given $\hat X$.
		Equality \eqref{eq:valid_2} holds because 
		$\hat x_j=0$ for $j\in V_1\setminus \hat X$ and $\hat x_j=1$ for $j\in \hat X$. Inequality \eqref{eq:valid_3} follows from the assumptions that $\bigcup_{k=1}^c C_2^k \subseteq V_2$, and $C_2^i \cap C_2^j = \emptyset$ for all $i,j = 1,\dots,c $ with $i \neq j$. Inequality \eqref{eq:valid_4} follows from $C_1^k \cap \{ V_1 \setminus \{ r(v) \} \} \subseteq V_1 \setminus \{ r(v) \}$ for $k = 1,\dots,c$ and $v \in V_2$. Inequality \eqref{eq:valid_5} follows from the definition of $r_k$ for each $k \in C''$. If $r_k \neq r(v)$ for some $k \in C''$ and $v \in C_2^k$, then there is no arc $(r_k,v)$ in $E_\omega$ and the value of $|v \cap \mathcal{U}_\omega (\{j, r_k\},\emptyset)|$ is equal to 0 for  $j \in C_1^k \cap \{ V_1 \setminus \{  r_k \} \}$. We obtain  equality \eqref{eq:valid_6} by reorganizing the terms in inequality \eqref{eq:valid_5}. Inequality \eqref{eq:valid_7} follows from $r_k \in C_1^k$  and $\mathcal{U}_\omega (\{j, r_k\},V_1 \setminus C_1^k) \subseteq \mathcal{U}_\omega (\{j, r_k\},\emptyset)$ for all $k \in C''$ and $j \in C_1^k \cap \{ V_1 \setminus \{  r_k \} \}$. Equality \eqref{eq:valid_8} follows from the assumption that each pair of distinct nodes $\{i,j\} \in C_1^k$ satisfies $|\mathcal{U}_\omega (\{i,j\}, V_1 \setminus C_1^k) \cap C_2^k| = n_\omega(C_1^k)$ for all $k=1,\dots,c$ and $n_\omega(C_1^k)\in \mathbb Z_+$. Equality \eqref{eq:valid_9} follows from $\hat x_{r_k} = 1$. We obtain equality \eqref{eq:valid_10} by reorganizing the terms in  inequality \eqref{eq:valid_9}. Inequality \eqref{eq:valid_11} follows from the assumption that    $\sum_{j \in C_1^k} x_j\le 1$ for each $k \in C'$. Finally, inequality \eqref{eq:valid_12} follows from $ \eta_\omega^k \ge 0, x_k\in\{0,1\}$, and $k\in D$. This completes the proof.


	\end{proof}

\begin{example}\label{ex:usual_ex}
	Let  $V_1 = \{1,2,3,4\}$ and  $V_2=\{1,2,3,4,5,6\}$. The bipartite graph associated with this example is depicted in \figref{Fig:usual_ex}.  Consider the parameters  for  inequality \eqref{eq:CIneq} given in \tabref{Table:usual_ex1}. The corresponding inequality   \eqref{eq:CIneq} is
	\begin{align}
			\theta_\omega \le 
		5+0x_1+0x_2+0x_3+x_4, \label{eq:dneempty}
	\end{align}		
	which is equivalent to a submodular inequality \eqref{eq:sub_cut} with $\bar X=\{1,2,3\}$.

	\begin{figure}[htbp]
		\centering	
		\scalebox{.8}{\begin{tikzpicture}[->,>=stealth',shorten >=1pt,auto,node distance=2.8cm,
		semithick]
		\tikzstyle{every state}=[fill=none,draw=black,text=black]
		
		\node[state] 		 (1)              {$1$};
		\node[state]         (2) [right of=1] {$2$};
		\node[state]         (3) [right of=2] {$3$};
		\node[state]         (4) [right of=3] {$4$};		
		\node[state]         (a) [below left of=1] {$1$};
		\node[state]         (b) [right of=a] {$2$};
		\node[state]         (c) [right of=b] {$3$};
		\node[state]         (d) [right of=c] {$4$};
		\node[state]         (e) [right of=d] {$5$};
		\node[state]         (f) [right of=e] {$6$};
		
		\path 
		(1) edge              node {} (b)
		(1) edge              node {} (c)
		(2) edge              node {} (a)
		(2) edge              node {} (b)
		(2) edge              node {} (c)
		(3) edge              node {} (d)
		(3) edge              node {} (f)
		(4) edge              node {} (d)
		(4) edge              node {} (e)
		;
		\end{tikzpicture}}
		\caption{An example of a bipartite graph with 4 sets and 6 items.}
		\label{Fig:usual_ex} 
	\end{figure}
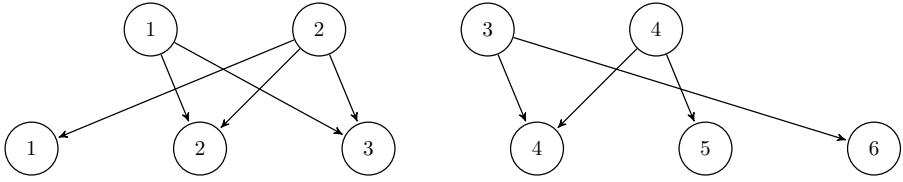

	
	\begin{table}[htb]
		\caption{The choice of parameters for  inequality \eqref{eq:CIneq} with $ D  \neq \emptyset$.}
		\label{Table:usual_ex1}	
		\begin{center}
			\begin{tabular}{ |l|l| }
				
				\hline
				$\{C_1^1, C_1^2\}$ & $D =\{d_1 , d_2\}$ \\ \hline
				$C_1^1 = \{1,2\}$ & $d_1  = \{2\}$  \\
				$n_\omega(C_1^1) = 2$       & $\eta_\omega^2 = 1$	\\
				$C_2^1 = \{2,3\}$ &   \\\hline
				$C_1^2 = \{3,4\}$ & $d_2 = \{3\}$  \\
				$n_\omega(C_1^2) = 1$       & $\eta_\omega^3 = 1$	\\
				$C_2^2 = \{4\}$ &   \\	
				\hline
				
			\end{tabular}
		\end{center}
	\end{table}
	
	If we let $D  = \emptyset$ for the same choices of $C_1^k, C_2^k, k=1,2$, then we obtain a facet-defining inequality $\theta_\omega \leq 3 + x_2 + x_3+x_4$, which is stronger than inequality \eqref{eq:dneempty}. In addition, this inequality cannot be generated as a submodular inequality \eqref{eq:sub_cut} for any selection of $\bar X$. 	
	We formalize this observation next.
\end{example}

	
	\begin{proposition}\label{prop:gen}
		Inequalities  \eqref{eq:CIneq} subsume the submodular inequalities  \eqref{eq:sub_cut}.
	\end{proposition}	
	\begin{proof}
		We show that a submodular inequality \eqref{eq:sub_cut}	for a given $\bar X \subseteq V_1$ can be represented as a corresponding inequality \eqref{eq:CIneq}.	
		To establish this correspondence, let  $D  = \bar X$. From the definition of $\eta_\omega^d $, for $d\in V_1$,  the term $\sum_{d \in D } \eta_\omega^d $ denotes the number of nodes reachable from only one node in $\bar X$.	
		Let $\bar C$ denote a set of nodes reachable from $\bar X$ and at least two nodes in $V_1$, and let $c=|\bar C|$. For $k\in \bar C$ let  $C_1^k=\{j\in V_1| (j,k)\in E_\omega\}$, i.e., $C_1^k$ is the set of all nodes that can reach $k\in \bar C$, and  let $C_2^k = \{k\}$ with $n_\omega(C_1^k)=1$. The term $\sum_{i=1}^{c} n_\omega(C_1^i)$ denotes the number of nodes reachable from $\bar X$ and at least two nodes in $V_1$.
		Next, we show that  inequality \eqref{eq:CIneq}  with this choice of   $D $,  $C_1^k$ and $C_2^k$ for all $k=1,\dots,c$ is equivalent to the submodular inequality.
		\begin{itemize}
			\item[(i)] For each $j \in V_1 \setminus \bar X$, the coefficient of $x_{j}$ in inequality \eqref{eq:CIneq} is $\sigma_\omega(\{j\}) - \sum_{i=1}^{c} \sum_{j \in C_1^i} n_\omega(C_1^i)$, where the second term equals the number of nodes reachable from $\bar X$ and $j$. Hence this coefficient   is equivalent to the marginal contribution term $\rho^\omega_{j} (\bar X)$.
			\item[(ii)] For each $j \in \bar X$, the coefficient of $x_j$ in inequality \eqref{eq:CIneq} is $\sigma_\omega(\{j\}) - \eta_\omega^j - \sum_{i=1}^{c} \sum_{j \in C_1^i} n_\omega(C_1^k)$, where the term $ \eta_\omega^j + \sum_{i=1}^{c}  \sum_{j \in C_1^i} n_\omega(C_1^i)$ equals the number of nodes reachable from $j$, i.e., $\sigma_\omega(\{j\}) $.  Hence, this coefficient  is   $0$.
			\item[(iii)] The right-hand side of inequality \eqref{eq:CIneq} is $\sum_{d \in D } \eta_\omega^d  +\sum_{i=1}^{c} n_\omega(C_1^i)$, which equals the number of nodes reachable from $\bar X$, i.e.,  $\sigma_{\omega}(\bar X)$. 
		\end{itemize}
		Hence, any submodular inequality can be represented as an inequality \eqref{eq:CIneq}. 
		Example \ref{ex:usual_ex} shows that there are inequalities \eqref{eq:CIneq} that cannot be written as submodular inequalities \eqref{eq:sub_cut}. This completes the proof.
	\end{proof}

Next, we provide a necessary condition for inequality \eqref{eq:CIneq} to be facet defining.

\begin{proposition}\label{prop:stronger}
	Inequality \eqref{eq:CIneq} is facet defining for conv($\mathcal S_\omega$) only if $D  = \emptyset$. 
	
\end{proposition}
\begin{proof}	
	We show that inequality \eqref{eq:CIneq} with $D=\emptyset$ given by 
	\begin{linenomath} 
	\begin{equation}\label{eq:CIneq_stronger}
		\theta_\omega \le \sum_{k = 1}^{c } n_\omega(C_1^k) \Big (1 - \sum_{j \in C_1^k}  x_j\Big ) + \sum_{j \in V_1} \sigma_\omega(\{j\})x_j
	\end{equation}
\end{linenomath} 	
	dominates inequality \eqref{eq:CIneq} with $D\ne \emptyset$. To see this, observe that the coefficients $n_\omega(C_1^k), k=1,\dots,c$ and $\sigma_\omega(\{j\}), j\in V_1$ do not depend on $D$. Hence, for the same choice of $C_1^k, C_2^k, k=1,\dots, c$, inequality  \eqref{eq:CIneq} with $D\ne \emptyset$ has the same terms as inequality  \eqref{eq:CIneq} with $D=\emptyset$, as well as the additional term $\sum_{k \in D } \eta_\omega^k (1 - x_k)\ge 0$, because $x_k\in \{0,1\}$ and $ \eta_\omega^k \ge 0$. Hence, we need to have $D  = \emptyset$ for inequality  \eqref{eq:CIneq} to be facet defining for conv($\mathcal S_\omega$).
\end{proof}

Note that we allow $D\ne \emptyset$ in the definition of inequality \eqref{eq:CIneq} to be able to show that inequality  \eqref{eq:CIneq}  subsumes submodular inequality  \eqref{eq:sub_cut}. However, we see from the necessary condition in Proposition \ref{prop:stronger} that it suffices to consider inequalities  \eqref{eq:CIneq}   with  $D=\emptyset$. 
Next we give some sufficient conditions for inequality \eqref{eq:CIneq_stronger} to be facet defining for conv($\mathcal S_\omega$).

\begin{proposition}\label{prop:facet}
	Inequality \eqref{eq:CIneq_stronger} is facet defining for conv($\mathcal S_\omega$) if  the following conditions hold: 
	\begin{itemize}
		\item[(i)]  $C_1^i \cap C_1^j = \emptyset$ for each  $i,j=1,\dots,c, i\ne j$,and 
		\item[(ii)] for each $k=1,\dots,c$, there exists at least one pair of nodes $\{i,j\} \in C_1^k$ such that  $\mathcal{U}_\omega (\{i,j\}, \emptyset) = \mathcal{U}_\omega (\{i,j\}, V_1 \setminus C_1^k)  \subseteq C_2^k$ and $\mathcal{U}_\omega (\{i,r\}, \emptyset) = \mathcal{U}_\omega (\{j,r\}, \emptyset) = \emptyset$ for all $r \in V_1 \setminus C_1^k$.	
	\end{itemize}	 
\end{proposition}
\begin{proof}	
	Note that for $\omega \in \Omega$, dim$(\mathcal S_\omega)=n+1$. We enumerate $n+1$ affinely independent points that are on the face defined by inequality \eqref{eq:CIneq_stronger} under  conditions (i) and (ii).

	Let a pair of nodes $\{f_1^k, f_2^k\} \in C_1^k$ be selected by  condition (ii) for all $k = 1,\dots,c $, where $f_1^k \neq f_2^k$, $\mathcal{U}_\omega (\{f_1^k, f_2^k\}, \emptyset) = \mathcal{U}_\omega (\{i,j\}, V_1 \setminus C_1^k) \subseteq C_2^k$, and $\mathcal{U}_\omega (\{f_1^k,r\}, \emptyset) = \mathcal{U}_\omega (\{f_2^k,r\}, \emptyset) = \emptyset$ for all $r \in V_1 \setminus C_1^k$.  	
	Let $\bar L = V_1 \setminus \bigcup_{k=1}^c C_1^k$. Based on condition (i), $\sum_{k=1}^{c }|C_1^k| + |\bar L| = n$, which means that each node $i \in V_1$ can only belong either to  $\bar L$  or  to one set $C_1^k$ for some $k=1,\dots,c $.  We describe  $n+1$ points on the face defined by inequality \eqref{eq:CIneq_stronger} next.
	
	Consider a point  $(\theta_\omega,x)^0 = ( \sum_{k=1}^{c }\sigma_\omega(\{f_1^k\})+ \sum_{k=1}^{c }\sigma_\omega(\{f_2^k\}) - \sum_{k=1}^{c }n_\omega(C_1^k) , \beta_0)$, where $\beta_0 = \sum_{k=1}^{c} \mathbf {e}_{f_1^k}+ \sum_{k=1}^{c} \mathbf {e}_{f_2^k}$. Recall that for all $k = 1,\dots,c $,  condition (i) ensures that  $\mathcal{U}_\omega (\{f_1^k, f_2^k\}, \emptyset) = \mathcal{U}_\omega (\{f_1^k, f_2^k\}, V_1 \setminus C_1^k) \subseteq C_2^k$, so that $n_\omega(C_1^k) = |\mathcal{U}_\omega (\{f_1^k, f_2^k\}, \emptyset)|$. Let $\bar S(x) = \{i \in V_1: x_i = 1\}$. Since $\mathcal{U}_\omega (\{f_1^k,r\}, \emptyset) = \mathcal{U}_\omega (\{f_2^k,r\}, \emptyset) = \emptyset$ for all $r \in V_1 \setminus C_1^k$, we have $\sigma_{\omega}(\bar S(\beta_0)) = \sum_{k=1}^{c } (\sigma_\omega(\{f_1^k\}) +  \sigma_\omega(\{f_2^k\}) -|\mathcal{U}_\omega (\{f_1^k, f_2^k\}, \emptyset)|) = \sum_{k=1}^{c }(n_\omega(C_1^k) +\sigma_\omega(\{f_1^k\}) - n_\omega(C_1^k)  + \sigma_\omega(\{f_2^k\}) - n_\omega(C_1^k))$, hence $(\theta_\omega,\beta_0)$ is on the face defined by inequality \eqref{eq:CIneq_stronger}.	
	
	For $i \in \bar L$, let $\beta_i^{\bar L}=\sum_{k=1}^{c} \mathbf {e}_{f_1^k}+ \sum_{k=1}^{c} \mathbf {e}_{f_2^k} + \mathbf {e}_{i}$. Consider the  point $(\theta_\omega,x)^i = ( \sum_{k=1}^{c }\sigma_\omega(\{f_1^k\})  + \sum_{k=1}^{c }\sigma_\omega(\{f_2^k\}) - \sum_{k=1}^{c }n_\omega(C_1^k) + \sigma_\omega(\{i\}), \beta_i^{\bar L} )$ for each $i \in \bar L$. Since $\mathcal{U}_\omega (\{f_1^k,r\}, \emptyset) = \mathcal{U}_\omega (\{f_2^k,r\}, \emptyset) = \emptyset$ for all $r \in V_1 \setminus C_1^k$, we have $\sigma_{\omega}(\bar S(\beta_i^{\bar L})) = \sigma_{\omega}(\bar S(\beta_0))+ \sigma_\omega(\{i\})= \sum_{k=1}^{c } \sigma_\omega(\{f_1^k\})+ \sum_{k=1}^{c } \sigma_\omega(\{f_2^k\})- \sum_{k=1}^{c }|\mathcal{U}_\omega (\{f_1^k, f_2^k\}, \emptyset)| + \sigma_\omega(\{i\}) = \sum_{k=1}^{c } (n_\omega(C_1^k) +\sigma_\omega(\{f_1^k\}) - n_\omega(C_1^k)  + \sigma_\omega(\{f_2^k\}) - n_\omega(C_1^k))+ \sigma_\omega(\{i\})$, hence $(\theta_{\omega}^i,\beta_i^{\bar L})$ for all $i \in \bar L$ are on the face defined by inequality \eqref{eq:CIneq_stronger}.


	Let $\bar C =\{1,\dots,c\}$. For $i \in C_1^k$, $k = 1,\dots,c $, let $\beta_i^k=\mathbf {e}_{i}+\sum_{j \in \bar C \setminus \{k\}} \mathbf {e}_{f_1^j} + \sum_{j \in \bar C \setminus \{k\}} \mathbf {e}_{f_2^j} $. Consider the  point
	$(\theta_\omega,x)^{ik} = (  \sum_{j\in \bar C \setminus \{k\}}(\sigma_\omega(\{f_1^k\}) + \sigma_\omega(\{f_2^k\}) - n_\omega(C_1^j)) + n_\omega(C_1^k) + \sigma_\omega(\{i\}) - n_\omega(C_1^k), \beta_i^k)$ for each $i \in C_1^k$ and $k = 1, \dots, c $.  For  $i \in C_1^k$,  $k = 1, \dots, c $,  condition (ii) ensures that $\mathcal{U}_\omega (\{i,f_1^j\}, \emptyset) = \mathcal{U}_\omega (\{i,f_2^j\}, \emptyset) = \emptyset$ for all $j = 1,\dots,c $ and $j \neq k$. We have $\sigma_{\omega}(\bar S(\beta_i^k)) = \sigma_{\omega}(\bar S(\beta_0)) - \sigma_\omega(\{f_1^k\}) -\sigma_\omega(\{f_2^k\}) + n_\omega(C_1^k) + \sigma_\omega(\{i\}) - n_\omega(C_1^k) =  \sum_{j\in \bar C \setminus \{k\}}  (n_\omega(C_1^j) + \sigma_\omega(\{f_1^j\}) - n_\omega(C_1^j) + \sigma_\omega(\{f_2^j\}) -  n_\omega(C_1^j)) + n_\omega(C_1^k) + \sigma_\omega(\{i\}) - n_\omega(C_1^k)$, hence $(\theta_\omega,\beta_i^k)^{ik}$ for all $i \in C_1^k$ and $k = 1, \dots, c $ are on the face defined by inequality \eqref{eq:CIneq_stronger}. These $1+|\bar L|+\sum_{k=1}^{c }|C_1^k| = n+1$ points are affinely independent.

\end{proof}

Algorithm \ref{alg:benders} describes a sampling-based method to solve PPSC by using inequalities \eqref{eq:sub_cut} or   \eqref{eq:CIneq} as  feasibility cuts. The proposed algorithm includes the Benders phase (Lines \ref{alg4:2}-\ref{alg4:14}) and the oracle phase (Lines \ref{alg4:15}-\ref{alg4:28}). Algorithm \ref{alg:benders} starts with a given set of feasibility cuts,  $\mathcal {\bar C}$. In the Benders phase, master problem \eqref{eq:PPSC-sub-Master} provides an incumbent solution $(\bar x, \bar \theta, \bar z)$ at each iteration (Line \ref{alg4:3}). \revised{For each scenario, the incumbent solution $(\bar x, \bar \theta, \bar z)$ is used for checking  feasibility (Line \ref{alg4:11}). If the condition  in Line \ref{alg4:11} is not satisfied, then   $(\bar x, \bar \theta, \bar z)$ is infeasible for the sample approximation problem. In this situation, we add a feasibility cut \eqref{eq:sub_cut} or \eqref{eq:CIneq} for  $\omega$ under the condition in Line \ref{alg4:6}. The Benders phase terminates when the condition in Line \ref{alg4:11} is satisfied.} The optimal solution given by the Benders phase to the sample approximation problem is then checked for feasibility with respect to the true distribution, by calling the subroutine FeasibilityCut($\bar x, \kappa,\bar{\mathcal C}$)  in the oracle phase   (Lines \ref{alg4:15}-\ref{alg4:28}). We use inequality \eqref{eq:LLcut_strong}  with  $\kappa(J_0)\le \kappa$ as the feasibility cut to cut off infeasible $\bar x$ until master problem \eqref{eq:PPSC-sub-Master} provides a truly feasible solution to the original (non-sampled) problem.

\begin{algorithm}[htb]\label{alg:benders}
	\SetAlgoLined
	Input: $\kappa\in\{1,2\}$. Start with $\mathcal {\bar C}=\{0\le  \theta_\omega\le m, \omega\in \Omega \}$ \label{alg4:1} \;
	\While{$True$}
	{ \label{alg4:2}
		Solve master problem \eqref{eq:PPSC-sub-Master} and obtain an incumbent solution $(\bar x, \bar \theta, \bar z)$ \label{alg4:3}\;		
		\If{\revised{$\sum_{\omega \in \Omega}\{p_\omega: \sigma_\omega (\bar x)\ge \tau \} \ge 1-\epsilon $}}
		{ \label{alg4:11}
			\bf{break}\;
		}\label{alg4:13}
		\Else
		{
			\For{$\omega\in \Omega$}
			{\do 
				\; 
				\If{\revised{$\tau  > \sigma_\omega (\bar x)$ and $\theta_\omega > \sigma_\omega (\bar x)$}}
				{ \label{alg4:6}
					Add a feasibility cut \eqref{eq:sub_cut} or   \eqref{eq:CIneq} to $\mathcal {\bar C}$ in master problem \eqref{eq:PPSC-sub-Master}\; \label{alg:step-cut}
				} \label{alg4:9}
			} \label{alg4:10}		
		}		
	} \label{alg4:14}

	\While{$\mathcal{A}(\bar x) < 1-\epsilon$}
	{ \label{alg4:15}
		Call  FeasibilityCut($\bar x, \kappa,\bar{\mathcal C}$)\;
		Solve master problem \eqref{eq:PPSC-sub-Master} and obtain an incumbent solution $\bar x$ \label{alg4:27} \;		
	}\label{alg4:28}

	Output  $\bar x$ as an optimal solution.
	\caption{Sampling-Based Delayed Constraint Generation Algorithm with a Probability Oracle for PPSC}
\end{algorithm}

For a given incumbent solution $\bar x$ with $\bar X = \{i \in V_1:\bar{x_i} = 1\}$,  we generate the corresponding violated submodular inequality \eqref{eq:sub_cut} as in \cite{first2016}, if infeasible.    	 
Next we describe how to generate a violated new valid inequality \eqref{eq:CIneq_stronger} with $D  = \emptyset$ (due to the necessary facet condition in Proposition \ref{prop:stronger}), in polynomial time for a given  infeasible solution. 	 
Consider the case that a node in $V_2$ is a common node for at least two nodes in $V_1$ and at least one node in $\bar X$. We find the set of nodes in $V_2$   reachable from at least two nodes in $V_1$ and at least one node in $\bar X$ by depth-first search, with the worst case complexity $\mathcal{O}(nm)$. Then, let $V_2'$ be a subset of nodes in $V_2$, where \revised{each $j \in V_2'$ is reachable from at least two nodes in $V_1$ and at least one node in $\bar X$}. For $k\in V_2'$ let  $V_k \subseteq V_1$ denote a set of nodes that $k \in \mathcal{U}_\omega (V_k, V_1 \setminus V_k)$. Note that $V_k$ can be obtained by solving a reachability problem to find which nodes in $V_1$ can reach node $k\in V_2'$. For each $k \in V_2'$, we let $C_1^k = V_k$ and $C_2^k = \{k\}$ with $n_\omega(C_1^k)=1$. The complexity of generating $C_1^k$ for all $k = 1,\dots, |V_2'|$ is $\mathcal{O}(m|V_2'|)$.  Thus, a violated inequality \eqref{eq:CIneq_stronger} can be generated in polynomial time. 

\revised{Finally, for PPSC with the linear threshold model, given a set of sampled scenarios, we observe that a polynomial number of submodular inequalities \eqref{eq:sub_cut} or  inequalities \eqref{eq:CIneq} is sufficient  to reach an optimal solution of the  master problem \eqref{eq:PPSC-sub-Master}.  In the following propositions, we summarize this result and its consequence of providing a compact MIP to solve PPSC with the linear threshold model. }

\revised{\begin{proposition}\label{prop:sample_LT}
	For PPSC with the linear threshold model, 
	adding the submodular inequalities \eqref{eq:sub_cut} with $\bar X=\emptyset $, which are equivalent to \eqref{eq:CIneq_stronger} for any choice of parameters, 
	to the set $\mathcal{\bar C} $ 
for all $\omega \in \Omega$ is sufficient	 to reach an optimal solution of the  master problem \eqref{eq:PPSC-sub-Master}.
	\end{proposition}    
\begin{proof}
	In the live-arc graph scenario generation method proposed by  Kempe et al.\ \cite{KKT03} for the linear threshold model, each node $j \in V_2$ has at most one incoming arc from a node $i \in V_1$ for each scenario $\omega \in \Omega$. Therefore, if $j \in V_2$ is reachable from $i \in V_1$, then $j$ is not reachable from any $i' \in V_1 \setminus \{i\}$. Therefore, 
	for any choice of $c, C_1^k , C_2^k, k=1,\ldots, c$ for inequality \eqref{eq:CIneq_stronger}, we must have  $n_\omega(C_1^k) = 0 $ for $k=1,\ldots, c$. In other words, there can be no common nodes in $V_2$ that are reachable from any two distinct nodes in $V_1$. Therefore,   
	 the submodular inequalities  \eqref{eq:sub_cut} with $\bar X=\emptyset $ are equivalent to the new inequalities \eqref{eq:CIneq_stronger} for any choice of parameters, and they are given by 
	 \begin{linenomath} 
\begin{equation} \label{eq:subempty}
\theta_{\omega} \le \sum_{i \in V_1} \sigma_{\omega}(\{i\})x_i.
\end{equation}
\end{linenomath} 
Any submodular inequality  \eqref{eq:sub_cut} with $\bar X\ne \emptyset$ given by
		\begin{linenomath} 
		\begin{align*}
		\theta_\omega &\le \sigma_\omega(\bar X) +\sum_{j\in V_1\setminus \bar X} \rho^\omega_j(\bar X) x_j  = \sum_{i \in \bar X} \sigma_{\omega}(\{i\}) + \sum_{j\in V_1\setminus \bar X} \sigma_{\omega}(\{j\})x_j,
		\end{align*}
	\end{linenomath} 
	is dominated by inequality \eqref{eq:subempty}. This completes the proof.
\end{proof}

\begin{proposition}\label{prop:sample_LT2}	
  For PPSC under the linear threshold model, given a set of scenarios $\Omega$, the  master problem formulation \eqref{eq:PPSC-sub-Master}, and the deterministic equivalent formulation \eqref{eq:sample-DEP} can be reduced to the following formulation in $(x,z)$-space
		\begin{subequations}\label{eq:DEP_LT}
			\begin{linenomath} 
				\begin{align}
				\min~~& \sum_{i \in V_1} b_i x_i \label{eq:DEP_LT1}\\
				\text{s.t.}~~& \sum_{i \in V_1} \sigma_{\omega}(\{i\})x_i \ge \tau z_{\omega} & {\omega}\in \Omega \label{eq:DEP_LT2} \\
				&\sum_{\omega \in \Omega} p_{\omega} z_{\omega}\ge 1-\epsilon  \\
				& x\in \mathbb{B}^n, z\in \mathbb{B}^{|\Omega|}. \label{eq:DEP_LT4}
				\end{align}
			\end{linenomath} 
		\end{subequations}		
\end{proposition}    
\begin{proof}
		In Proposition \ref{prop:sample_LT}, we show that adding inequalities \eqref{eq:subempty} to the master problem  \eqref{eq:PPSC-sub-Master} as  feasibility cuts is sufficient to capture the submodular coverage function $\sigma_\omega(x)$. Then, the $\theta_\omega$ variables can be projected out from the formulation using inequalities \eqref{eq:subempty} and \eqref{eq:PPSC-sub-Master-3}, leading to inequalities \eqref{eq:DEP_LT2} and the formulation \eqref{eq:DEP_LT}.
		
		Next we show that the deterministic equivalent formulation  \eqref{eq:sample-DEP} can be reduced to  formulation \eqref{eq:DEP_LT} in $(x,z)$-space for the linear threshold model. 
		From the definition of $t_{ij}^\omega$ for all $i \in V_1$, $j \in V_2$ and $\omega \in \Omega$, where $t_{ij}^\omega=1$ if arc $(i,j)\in E_\omega$ for $\omega\in \Omega$, and $t_{ij}^\omega=0$ otherwise, we have $\sigma_\omega(\{i\}) = \sum_{j \in V_2} t_{ij}^\omega$ for all $i \in V_1$ and $\omega \in \Omega$, because there is only one incoming arc to node $j\in V_2$ in every scenario $\omega\in \Omega$ in a linear threshold model \cite{KKT03}. 
	In formulation \eqref{eq:sample-DEP},  summing the constraints \eqref{eq:sample-DEP2} over all $i\in V_2$, we obtain 
		 $\sum_{i \in V_2}\sum_{j\in V_1}  t_{ij}^{\omega} x_j \ge \sum_{i \in V_2} y_i^{\omega}, \forall \omega \in \Omega $, which is equivalent to
		 \begin{linenomath}  
		 \begin{equation}\label{eq:projy}
		 \sum_{j\in V_1}  \sigma_\omega(\{j\}) x_j \ge \sum_{i \in V_2} y_i^{\omega}, \forall \omega \in \Omega.
		 \end{equation} 
		 \end{linenomath} 
	Now 	 we can project out the $y$ variables using the constraints \eqref{eq:projy} and \eqref{eq:sample-DEP3}, and obtain the constraints  \eqref{eq:DEP_LT2} and the formulation \eqref{eq:DEP_LT}. 
		 This completes the proof.
\end{proof}}

	\section{Computational Experiments}\label{sec:computation}

In this section, we report our experiments with PPSC to
	demonstrate the effectiveness of our proposed methods. All methods are implemented in C++ with IBM ILOG CPLEX 12.7 Optimizer. All experiments were executed on a Windows 8.1 operating system with an Intel Core i5-4200U 1.60 GHz CPU, 8 GB DRAM, and x64 based processor. \revised{For the master problem of the decomposition algorithms and the deterministic mixed integer programming models, we specify the MIP search method as traditional branch-and-cut with the lazycallback function of CPLEX. We set the number of threads to one. CPLEX presolve process is turned off for the traditional branch-and-cut for solving the decomposition algorithms.} The relative MIP gap tolerance of CPLEX is set to the default value, so a feasible solution which has an optimality gap of $10^{-4}$\% is considered optimal. The time limit is set to one hour.

\revised{Our dataset is motivated by human sexual contact network (human interaction network) introduced in \cite{Guler2002,Newman2003}. This class of social networks is represented as a bipartite graph, where $V_1$ and $V_2$ denote the groups of different genders and arcs denote the connections between males and females. Note that, in this context, it is natural to assume that $|V_1|$ is approximately equal to $|V_2|$. }

We generate a complete bipartite graph with arcs from all nodes $ i \in V_1$ to all nodes $j \in V_2$. We partition the nodes in $V_1$ into two sets $V_1^1$ and $V_1^2$, where each node $i \in V_1^1$ can cover a higher expected number of items than each node $j \in V_1^2$. Our computational experiments include two parts. \revised{ In the first part of our computational study, we test the exact delayed constraint generation algorithm given in Algorithm \ref{alg:GDCG}, and the sampling-based approach described in  Algorithm \ref{alg:benders} to solve PPSC under the independent probability coverage model.  In the second part, we compare the exact delayed constraint generation algorithm and the deterministic equivalent MIP formulation \eqref{eq:PPSC-LTMIP} to solve PPSC under the linear threshold model. In addition, compact MIP model \eqref{eq:DEP_LT} described in Proposition \ref{prop:sample_LT2} is applied to PPSC under the linear threshold model.}

\subsection{PPSC under the Independent Probability Coverage Model}

In this subsection, we report our experiments with the independent probability coverage model. Recall that for the independent probability coverage model, $P(x,i) = 1- \prod_{j \in V_1} (1-a_{j,i}x_j)$ is used for calculating $\mathcal{A}(x)$, where $a_{u,i}$ denotes an independent probability that the set $u$ can cover the item $i$ with probability $a_{u,i}$. Because the corresponding model \eqref{eq:PPSC-NLP} is highly nonlinear, we do not attempt to solve it for the independent probability coverage model. 
We generate a complete bipartite graph where  each arc $(i, j)$ is assigned an independent probability $a_{ij}$ of being live  for all nodes $ i \in V_1, j \in V_2$. We consider the case that the expected number of covered items for each $i \in V_1^1$ is  $20\% \pm 2\%$ of the total number of nodes in $V_2$, and the expected number of covered items for each $i \in V_1^1$ is  $2\% \pm 2\%$ of the total  number of nodes in $V_2$. In particular, we let $a_{ij}  = 0.18 + i \times (0.22-0.18)/|V_1^1|$ for  each $i \in V_1^1$, and $a_{ij} = (i-|V_1^1|) \times (0.04)/|V_1^2|$ for each $i \in V_1^2$, where we let $V_1^2=\{|V_1^1|+1,\dots,n\}$. The size of bipartite graphs is $|V| \in \{60,90,120\}$. Unless otherwise noted, we let $ n = m = |V|/2 $ and $n = |V_1^1|+|V_1^2|$.  
We let $|V_1^1| = 10$ for all instances, and $|V_1^2| = n - 10$. We set the target $\tau  = 0.6m$. The risk level is set  as \revised{$\epsilon \in \{0.0125,0.025,0.05\}$. The objective function coefficients are set as $b_i \in [1,\bar b]$ for each $i \in V_1$, where $\bar b = \{1,100\}$. Note that for $\bar b \neq 1$, we set $b_i =i \bar b/ |V_1^1|$ for $i \in V_1^1$. Since each node $j \in V_1^2$ covers a fewer expected number of items than each node $i \in V_1^1$, we set a lower cost range for $j \in V_1^2$ compared to  $i \in V_1^1$, where $b_j \in [1,\bar b/2]$ and $b_j =(|V_1|-j-1)  \bar b/ (2|V_1^2| )$ for $j \in V_1^2$. }

We first solve PPSC under the independent probability coverage model exactly by using Algorithm \ref{alg:GDCG}, which is referred to as ``Oracle". To show the effect of the choice of   $\kappa(J_0)$ in inequality \eqref{eq:LLcut_strong} on the convergence of the algorithm, we study two cases of Oracle depending on the choice of the input parameter $\kappa$, i.e.,  Oracle ($\kappa= 1$) and Oracle ($\kappa = 2$). \tabref{Table:IC_Oracle} provides the comparison between the two methods,  column ``Cuts" denotes the total number of user cuts added to the master problem and column ``Time" denotes the solution  time  in seconds.

 \tabref{Table:IC_Oracle} shows that using the stronger no-good cuts (i.e., Oracle ($\kappa = 2$))  
drastically reduces the solution time and the number of cuts required when compared to the traditional no-good cuts (i.e., 
 Oracle ($\kappa = 1$)). None of the instances can be solved within the time limit if the traditional no-good cuts are used, whereas all instances are solved in less than six minutes with the coefficient strengthening for instances with $|V|=60$, and within 20 minutes for unit-cost instances with $|V|=90$.    Hence it is worthwhile to expend additional computational effort to strengthen inequality \eqref{eq:LLcut_strong} by using a larger right-hand side ($\kappa(J_0) = 2$ versus $\kappa(J_0) = 1$).  
\revised{ We observe that the instances with non-unit costs (i.e., $\bar b=100$) are harder to solve than instances with  unit cost (i.e., $\bar b=1$). For $|V| \ge 90$, none of the non-unit cost instances can be solved within the time limit. To test the limitation of Oracle ($\kappa = 2$), we also tested  instances with a larger size $|V| > 90$ with the same parameters $\epsilon$ and $\bar b$ as in \tabref{Table:IC_Oracle}. The solution time grows exponentially as $|V|$ increases for Oracle ($\kappa = 2$). For example, Oracle ($\kappa = 2$) can  solve only one instance with $(\bar b,\epsilon)=(1,0.05)$ within the time limit for $|V|=120$ with 3283 seconds. Based on our experience, the oracle-based exact method can solve instances with  unit cost of up to 100 nodes within an hour. }

\revised{
\begin{table}[htb]
	{ \caption{Oracle ($\kappa = 1$) vs.\ Oracle ($\kappa = 2$) for PPSC with the independent probability coverage model.}
		\label{Table:IC_Oracle}
		\begin{center}
			\scalebox{0.95}{\begin{tabular}{ 			 |p{1cm}p{1cm}p{1cm}||p{1.5cm}p{1.5cm}||p{1.5cm}p{1.5cm}|}
					\hline
					
					& & & 
					\multicolumn{2}{c||}{Oracle ($\kappa = 1$)} &
					\multicolumn{2}{c|}{Oracle ($\kappa = 2$)} \\
					
					\cline{4-7}%
					
					{$|V|$}  &  { $\bar b$} & { $\epsilon$}
					& {Time } & {Cuts }  & {Time }& { Cuts } 
					\\
					\hline 				
					\multirow{6}{4em}{60} & \multirow{3}{4em}{1}  & 0.0125 & $\ge3600$& 39844 & 90  & 4357 \\	
					&   & 0.025 & $\ge3600$  & 44387 & 12& 1404 \\ 				
					&   & 0.05 & $\ge3600$ & 51004 & 15  & 1295  \\

					& \multirow{3}{4em}{100}  & 0.0125 & $\ge3600$ & 49158 & 292  & 6903 \\ 				
					&   & 0.025 & $\ge3600$& 46392 & 320  & 8075 \\	
					&   & 0.05 & $\ge3600$& 49895  & 257 & 7161 \\ 		
					\hline
					\multirow{6}{4em}{90} & \multirow{3}{4em}{1}  & 0.0125 & $\ge 3600$& 45913 & 49  & 2758 \\	
					&   & 0.025 & $\ge 3600$  & 45627 & 387& 8072 \\ 				
					&   & 0.05 & $\ge 3600$ & 46594 & 1209  & 12894  \\
					& \multirow{3}{4em}{100}  & 0.0125 & $\ge 3600$  & 42648 & $\ge 3600$  & 25178\\ 				
					&  & 0.025 & $\ge 3600$  & 42351 & $\ge 3600$  & 25253 \\	
					&  & 0.05 & $\ge 3600$  & 40234 & $\ge 3600$& 24900  \\ 						 				
					\hline%

			\end{tabular}}
	\end{center}}
\end{table}	
}

To solve the problem for  networks with larger sizes (i.e., $|V|>100$), we consider the sampling-based approach that exploits the submodular substructure of PPSC. We demonstrate the usage of oracle for checking and fixing the feasibility of the solution given by the sampling-based approach. \revised{ First, we run the sampling-based methods on the instances with $|V|=60$ used in \tabref{Table:IC_Oracle} so that we can compare the feasible solution obtained at the end of the sampling-based method with the truly optimal solution obtained by the exact method.  
In our preliminary computational study, consistent with our observations with the exact method, we see that the instances with $\bar b = 100$ are harder to solve than the instances with $\bar b = 1$. Hence,  for the instances with $\bar b = 1$, we generate $|\Omega| = \{100,500,1000\}$ equiprobable scenarios, and for  $\bar b = 100$, we generate fewer equiprobable  scenarios ($|\Omega| = \{100,250,500\}$).} For each combination of $(|V|,\bar b,\epsilon,|\Omega|)$, we create three replications of the scenario set and report the average statistics.  We consider the sampling-based delayed constraint generation method (Algorithm \ref{alg:benders}), which is referred to as ``DCG" in this subsection. Recall that Algorithm \ref{alg:benders} is executed in two phases, the Benders phase, and the oracle phase. In the Benders phase, we apply two types of feasibility cuts, submodular inequality \eqref{eq:sub_cut} (referred to as DCG-Sub) and new valid inequality \eqref{eq:CIneq_stronger} (referred to as DCG-NV), to RMP \eqref{eq:PPSC-sub-Master}. In the oracle phase, we check whether the optimal solution to  the sample approximation problem, $\bar x$, obtained at the end of the Benders phase of DCG, is feasible for the original problem, by using the polynomial-time DP described in Section \ref{sec:dp}.  We use Algorithm \ref{alg:nos_PSC} with $\kappa=2$ in these experiments to add feasibility cuts  \eqref{eq:LLcut_strong} to RMP \eqref{eq:PPSC-sub-Master}.  We also consider the deterministic equivalent problem \eqref{eq:sample-DEP} using the linear representation of the chance constraint (referred to as DEP  \eqref{eq:sample-DEP}). In the case of DEP  \eqref{eq:sample-DEP}, once the sample approximation problem is solved to obtain an optimal solution  $\bar x$ to the sample approximation problem,  we also enter an oracle phase, where we check feasibility by using the polynomial-time DP described in Section \ref{sec:dp} as an oracle. If the current solution is not feasible with respect to the true distribution, then we add inequality \eqref{eq:LLcut_strong} with $\kappa(J_0) \le 2$ to the corresponding deterministic equivalent formulation and re-solve. We repeat this process until a feasible solution is obtained. \revised{The feasible solution  obtained at the end of the oracle phase may not be optimal with respect to the true distribution. For the instances for which the truly optimal solution is available (from \tabref{Table:IC_Oracle}), we provide information on the optimality gap of the  feasible solution. }

\revised{

	\begin{sidewaystable}[htp]
	{ \caption{Networks with $|V|=60$ for PPSC with the independent probability coverage model-Sampling.} 
		\label{Table:IC-table1}
		\centering 
		\scalebox{0.7}{\begin{tabular}{ |p{0.5cm}p{0.7cm}p{0.7cm}||p{1cm}p{1cm}p{1cm}p{1.2cm}|p{0.5cm}p{1cm}p{0.6cm}p{0.5cm}p{1.2cm}|p{1cm}p{1cm}p{1cm}p{1.2cm}|p{0.6cm}p{1cm}p{0.7cm}|p{1cm}p{1cm}p{1.2cm}|p{0.5cm}p{1cm}p{0.6cm}|  }
				\hline
				
				&	& & \multicolumn{9}{c|}{   DCG-NV } & \multicolumn{7}{c|}{   DCG-Sub } & \multicolumn{6}{c|}{ DEP \eqref{eq:sample-DEP} } 		
				\\		
				\cline{4-25}%
				&	& & \multicolumn{4}{c|}{   Master}& \multicolumn{5}{c|}{   Oracle } & \multicolumn{4}{c|}{  Master }& \multicolumn{3}{c|}{   Oracle } & \multicolumn{3}{c|}{  DEP }& \multicolumn{3}{c|}{  Oracle  } \\

				{$\bar b$} &	{$\epsilon$} & { $|\Omega|$}  
				& {Time(u) }& {Cuts} & {Nodes}  & {Gap(\%)} & {Inf} & {Time(u)} & {Cuts} & {nOpt} & {oGap(\%)}  &   {Time(u) }& {Cuts} & {Nodes}  & {Gap(\%)} & {Inf} & {Time(u)} & { Cuts } &   {Time(u) } & {Nodes}  & {Gap(\%)} & {Inf} & {Time(u)} & { Cuts }
				\\				
				\hline
				\multirow{9}{4em}{1}&\multirow{3}{4em}{0.0125}&100  &$\le 1$  & 369& 506 &0&3& 4 & 148&0&0 &    2&601&871&0&3&5&121&  21&1304&0&3&34& 102 \\
				& & 500 &	15 & 1010 & 4648 & 0 & 3 & 50 & 110&0&0 &  36 &2034& 8672&0&3&30&81 &	462& 3070&0&3&2441(2)& 9\\
				& & 1000 & 473 & 2408 & 74345 &0 & 1 & 110 &119&0&0 &	909&4650& 131859 &0& 1 &95&56& 2794(2)&8244&27.78&1&(1)&3	\\
				& \multirow{3}{4em}{0.025} & 100 & 2 & 404 &1079&0& 2 & 2 & 16&0&0& 3& 600 &1874& 0& 2& 3 & 30&  16&1006&0&2&63&27 \\
				&  & 500 &110& 1410&31559& 0 & 2 & 27& 1&0&0&  201& 2314&45543&0&1&31&1&	1386&5780& 0&1&2489&9\\
				&  & 1000 & 1678& 2565& 218824&0& 1& 290 & 1&0&0&  1563&4612&231118&0 &3&124&3&2012(1)&3687&30.83&0&-&-\\
				& \multirow{3}{4em}{0.05} & 100 & 5& 451& 3301 &0& 3&24 &135&0&0& 9& 638& 4954& 0 &3&25&142& 26&1123&0&3&233&34\\
				&  & 500 &1985& 2420&689696& 0& 0& -&-&0&0&1301(1)& 4189& 510209&23.2 &0 &-&-&1507&8646&0&0&-&- \\
				&  & 1000&1633(2) & 4553&567985& 26.3& 0& - & -&0&0&(3) & 7586&323825 & 36.35&-&-&- &2050(1)&3996&34.21&0&-&- \\\cline{1-25}
				
				\multirow{9}{4em}{100}&\multirow{3}{4em}{0.0125}&100&$\le 1$&318&348&0&3 &3&119&1&1.01&    $\le 1$&496&710&0&3&3&150 &12&581&0&3& 31& 41 \\
				& & 250&5&550&1795&0&3&6&56&0&0&  9&851&3447&0&3&12&88 &68&1831&0&3&128&23\\
				& & 500&37&954&10406&0&3&24&47&0&0&  78&1495&19342&0&3&16&74&499 &4288&0&3&1636&14 \\
				& \multirow{3}{4em}{0.025} & 100&$\le 1$&328&802&0&3 &2&5&2&5.74&  2&481&1537&0&2&2&29& 20&1010&0&2&40&9\\
				&  & 250 &16&620&6439&0&3 &11&3&2&3.45& 24&1009&10158&0&3&7&5&252 &5271&0&3&427&4\\
				&  & 500& 190&1077&46666&0&2 &19&5&2&3.45& 291&1778&79243&0&2&19&6& 928&6354&0&2&2895&1\\
				& \multirow{3}{4em}{0.05} & 100 &6&347&3828&0&3 &10&63&1&3.57& 8&584&6376&0&3&11&56&34 &1432&0&3&175&24\\
				&  & 250 &147&701&59203&0&2&31&38&1&1.19& 130&1171&61976&0&2&21&38&499 &9027&0&2&1173&13\\
				&  & 500 & 1463&1572&468709&0&2&43&7&1&1.19& 2757&2464&756766&0&2&95&11&1992(2) &11133&18.04&1&(1)&1\\
				\hline

	\end{tabular}}}
\end{sidewaystable}

In \tabref{Table:IC-table1}, we compare the performances of DCG-NV, DCG-Sub, and DEP \eqref{eq:sample-DEP} and demonstrate the utility of the oracle in the sampling-based delayed constraint generation algorithm for instances for which we are able to find the optimal solution to the true problem with the exact method.  Column ``Master" reports the statistics pertaining to the Benders phase of DCG, and column ``DEP" denotes the deterministic equivalent problem \eqref{eq:sample-DEP}.  Column ``Oracle" reports the statistics pertaining to the  oracle phase of DCG and  DEP  \eqref{eq:sample-DEP}. Note that we set a one hour time limit for both Master and DEP, and for  the oracle phase. In ``Master,"  ``DEP" and ``Oracle" columns, ``Time(u)" denotes the solution time in seconds and notation ``(u)" denotes the number of unsolved instances out of the three instances tested for the corresponding setting. In ``Master," column ``Cuts" denotes number of inequalities \eqref{eq:sub_cut} and \eqref{eq:CIneq_stronger} added to RMP \eqref{eq:PPSC-sub-Master} for DCG-Sub  and DCG-NV, respectively. In ``Oracle," column ``Cuts" denotes number of inequalities \eqref{eq:LLcut_strong} added to RMP \eqref{eq:PPSC-sub-Master} in the oracle phase. Column ``Nodes" denotes the number of branch-and-bound nodes traced in the Benders phase. For the instances that do not solve within the time limit, column ``Gap" reports the end gap given by $(ub-lb)/ub$, where $ub$ is the objective function value of the best feasible integer solution obtained within the time limit and $lb$ is the best lower bound available within the time limit for the sample approximation problem. We use the oracle phase to check and fix the infeasibility of the solution provided by the Benders phase.  Column ``Inf" denotes the number of instances, among the instances for which an optimal solution was found in the Benders phase, that provides a solution detected to be infeasible by the oracle. Note that we do not enter the oracle phase unless an optimal solution is found by the Benders phase. Hence, if none of the instances are solved to optimality in the Benders phase, we put a dash (-) under the relevant statistics of the oracle phase. In addition, if all solutions found at the end of the Benders phased are deemed feasible by the oracle phase (indicated by Inf = 0), then we put a dash (-) under the ``Time(u)" and ``Cuts" columns, because the time to confirm that the given solution is feasible is negligible and no oracle cuts are added in this case.    If the oracle phase  cannot be completed within the one-hour time limit due to the multiple MIPs that need to be solved after detecting infeasibility and adding no-good cuts, then we report    the number of unsolved instances out of the  number of instances tested in the oracle phase (given by ``Inf") in parentheses ``(u)." Recall that for the instances with $|V|=60$, we are able to obtain the truly optimal solution from the exact method. Thus, for these instances, we are able to calculate the actual gap between the objective function value of the feasible solution given by the sampling-based approach and the truly optimal value given by the exact method. Column ``nOpt" denotes the number of instances out of three that do not have the same optimal objective value obtained from the exact method. Column ``oGap" denotes the optimality gap between the optimal value given by DCG and the true optimal value given by Oracle ($\kappa = 2$) calculated as $100|(v-v^*)|/v$, where 
 $v$ is the objective function value of the feasible solution obtained from DCG and $v^*$ is the optimal objective function of the truly optimal solution found in \tabref{Table:IC_Oracle}.  We observe that if an instance is solvable within the time limit, all three methods, DCG-NV, DCG-Sub and  DEP  \eqref{eq:sample-DEP}, provide the same objective value at the end of the oracle phase. Therefore, we only show ``nOpt" and ``oGap" for DCG-NV, which is able to solve  most of the instances within the time limit.

First, we compare Tables \ref{Table:IC_Oracle} and \ref{Table:IC-table1}. We note that for  instances with $|V|=60$, the exact method using the true distribution can solve most problems faster than the sampling-based method that approximates the true problem with the  number of scenarios $|\Omega|\ge 500$. For example, the  solution time for the setting with $|V|=60, \bar b=1, \epsilon=0.05$
with the exact method, Oracle ($\kappa = 2$), is 15 seconds, but the average solution time with the sampling-based method DCG-NV  is 1985 seconds for 500 scenarios and two instances hit the time limit for 1000 scenarios. These results demonstrate that, for smaller networks, an exact method may be able to solve the problem to true optimality more efficiently  than a sampling-based method, which is not able to guarantee optimality. 
}

\revised{Next, we compare the performance of DCG-NV, DCG-Sub, and DEP \eqref{eq:sample-DEP} for problems with $|V|=60$. \tabref{Table:IC-table1}  shows that the solution time increases as  $\epsilon$ and $|\Omega|$ increase. We also note that the problems with non-unit costs are generally harder to solve.  Comparing the solution times, we observe that both versions of DCG (with submodular cuts, or with the new valid inequalities) are generally faster than DEP. In addition, for the instances that DCG can provide an optimal solution within the time limit, DCG-NV is faster than DCG-Sub in most cases. In addition, the columns ``Cuts" and ``Nodes" show that DCG-NV adds fewer user cuts and traces fewer branch-and-bound nodes than DCG-Sub in most cases. As a result,  inequality \eqref{eq:CIneq_stronger}, which we prove to be a   stronger inequality than inequality \eqref{eq:sub_cut} (Proposition \eqref{prop:stronger}), improves the computational performance of DCG.}

\revised{Next, we demonstrate the usage of oracle for checking and fixing the feasibility of the solution given by the sampling-based approach. In Oracle, Inf $=0$ denotes that there the solution provided at the end of the Benders phase is feasible. The instances that require feasibility cuts (indicated by a positive number in the Inf column) show that although the optimal solution for the sample approximation problem, provided by the Benders phase in DCG-Sub or DCG-NV or by DEP  \eqref{eq:sample-DEP}, is not feasible with respect to the true distribution, the oracle phase fixes the infeasibility in most cases. There are two cases for the DEP  \eqref{eq:sample-DEP} based method where the oracle phase hits the time limit and hence cannot provide a feasible solution.}  As expected, a larger number of scenarios better represents the true distribution and in general leads to an increased number of feasible solutions that do not require the oracle phase, although there are exceptions. \revised{ In general, if an instance has a large number of scenarios and an infeasible solution is detected, then the oracle phase spends more time on finding a feasible solution. There is no obvious trend between risk level $\epsilon$ and the number of added oracle cuts. 

Regarding the optimality gap due to solving a sample approximation problem instead of the true problem, from nOpt and oGap columns, comparing the optimal solutions provided by the exact method  with those of the sampling-based method  for the test instances in Tables  \ref{Table:IC_Oracle} and \ref{Table:IC-table1}, we observe that  the instances with unit cost attain true optimality with zero oGap in \tabref{Table:IC-table1} at the end of the oracle phase. In most cases, the solution found at the end of the Benders phase is not feasible and it is corrected during the oracle phase, which leads to feasible solutions. While the sampling-based method  only guarantees a feasible solution to the original problem at the termination of the oracle phase, in this set of experiments, all solutions for the unit cost instances   turn out to be optimal. We suspect that this 0\% gap occurs because the objective function has
		the same coefficient  for all the variables, and so there is a large number of solutions - some feasible, some infeasible - with the same objective function value.   For the instances with $\bar b=100$, we observe that oracle plays an important role in fixing the infeasibility in all instances and only two settings $(\bar b, \epsilon, |\Omega| ) = (100,0.0125,250)$ and $(100,0.0125,500)$ have zero oGap. In most cases, the feasible solutions for non-unit cost instances are suboptimal with an optimality gap of up to 5.74\%.
}

\revised{In \tabref{Table:IC-table2}, we report our experiments with $|V|=120$ for which we are not able to obtain a truly optimal solution using the exact method we proposed. Comparing Tables \ref{Table:IC-table1} and \ref{Table:IC-table2}, we see that the problems are harder for the sampling-based method  when the number of nodes increases.   The solution time and the number of branch-and-bound nodes increase drastically as $|V|$ increases for  most instances. For example, the setting $(|V|,\bar b, \epsilon, |\Omega| ) = (60,100,0.025,500)$ takes 190 seconds  and 46666 branch and bound nodes to solve on average, whereas the setting $(|V|,\bar b, \epsilon, |\Omega| ) = (120,100,0.025,500)$ takes   1581 seconds and 277612 branch-and-bound nodes on average.  \tabref{Table:IC-table2} shows the same trend as \tabref{Table:IC-table1}, where  the Benders phase of DCG  outperforms DEP with respect to solution time. For these instances, DEP \eqref{eq:sample-DEP} cannot solve half of the instances within the time limit, whereas both DCG-NV and DCG-Sub solve most instances to optimality. In addition, DCG-NV runs faster than DCG-Sub. We also demonstrate the usage of oracle for checking and fixing the feasibility of the solution given by Master and DEP. The results show that in  almost all  settings, except for three,    an infeasible solution is provided by the sampling-based approach. Therefore, it is important to use the oracle phase to fix the infeasibility. The number of infeasible solutions decreases as the number of scenarios increases, however, increasing sample size to reduce the infeasibility issues slows down the solution time of both the Benders and the oracle phases of DCG. 
As a result, there is a tradeoff between the solution time and the solution accuracy. 
	
	\begin{sidewaystable}[htp]
	{ \caption{Networks with $|V|=120$ for PPSC with the independent probability coverage model-Sampling.} 
		\label{Table:IC-table2}
		\centering 
		\scalebox{0.7}{\begin{tabular}{ |p{0.5cm}p{0.7cm}p{0.7cm}||p{1cm}p{1cm}p{1cm}p{1.2cm}|p{0.5cm}p{1cm}p{0.6cm}|p{1cm}p{1cm}p{1cm}p{1.2cm}|p{0.5cm}p{1cm}p{0.6cm}|p{1cm}p{1cm}p{1.2cm}|p{0.5cm}p{1cm}p{0.6cm}|  }
				\hline
				
				&	& & \multicolumn{7}{c|}{   DCG-NV } & \multicolumn{7}{c|}{   DCG-Sub } & \multicolumn{6}{c|}{ DEP \eqref{eq:sample-DEP} } 		
				\\		
				\cline{4-23}%
				&	& & \multicolumn{4}{c|}{   Master}& \multicolumn{3}{c|}{   Oracle } & \multicolumn{4}{c|}{  Master }& \multicolumn{3}{c|}{   Oracle } & \multicolumn{3}{c|}{  DEP }& \multicolumn{3}{c|}{  Oracle  } \\

				{$\bar b$} &	{$\epsilon$} & { $|\Omega|$}  
				& {Time(u) }& {Cuts} & {Nodes}  & {Gap(\%)} & {Inf} & {Time(u)} & { Cuts } &   {Time(u) }& {Cuts} & {Nodes}  & {Gap(\%)} & {Inf} & {Time(u)} & { Cuts } &   {Time(u) } & {Nodes}  & {Gap(\%)} & {Inf} & {Time(u)} & { Cuts }
				\\				
				\hline
				\multirow{9}{4em}{1}&\multirow{3}{4em}{0.0125}&100  & $\le 1$ & 695&604 & 0& 2&7&101& $\le 1$&871&885&0&2&7&114 &43&1099&0&2&72&8\\
				& & 500& 180&4069&29550&0&0&-&-&255&6187&31508&0&0&-&-&(3)&4645&26.93&-&-&-\\
				& & 1000 & 859&6083&84757&0&0&-&-&1740&9969&126428&0&0&-&-&(3)&1514&36.83&-&-&-\\
				& \multirow{3}{4em}{0.025} & 100 &3&836&1192&0&3&13&114&3&849&1496&0&3&21&134& 39&620&0&3&171&17\\
				&  & 500 & 866& 7084&100719&0&0&-&-&1206&9151&114554&0&0&-&-&(3)&5270&31.35&-&-&-\\
				&  & 1000 &(3)&11827&128201&39.59&-&-&-&(3)&16492&83493&42.58&-&-&-&(3)&1281&37.37&-&-&-\\
				& \multirow{3}{4em}{0.05} & 100 &17&954&7141&0&3&18&1&18&1282&8112&0&2&13&6& 72 &1033&0&3&77&7\\
				&  & 500  & 2426&4138&246194&0&2&1674&14&2461(1)&5139&354775&20&1&895&1&2415&2079&0&3&3110(1)&3\\
				&  & 1000 &2650(2)& 8669&192779&30.61&1&(1)&5&(3)&9952&225870&27.45&-&-&-&(3)&955&36&-&-&-\\\cline{1-23}

				\multirow{9}{4em}{100}&\multirow{3}{4em}{0.0125}&100&2&778&843&0&3&11&479&    3&1070&1372&0&3&20&608 &38 &1166&0 &3&168&72\\
				& & 250& 23&1567&6141&0&2&46&26& 47&2490&9305&0&2&106&100& 549 &3439&0&2&1325&11 \\
				& & 500&  162&2348&27594&0&3&203&39& 326&3823&42848&0&3&413&43& (3) &5492&12.61&-&-&-\\
				& \multirow{3}{4em}{0.025} & 100&  4&872&1944&0&3&18&233& 5&1058&2353&0&3&40&503& 36 &1002&0&3&220&65\\
				&  & 250 & 152&1634&44208&0&2&557&117& 199&2367&49445&0&2&600&217&1477(1) &9610&7.12&2&1723(1)&12\\
				&  & 500&  1581&2671&277612&0&2&1806(1)&35& 2561&4533&315038&0&2&1591&23&(3) &4732&19.87&-&-&-\\
				& \multirow{3}{4em}{0.05} & 100 &17&865&8679&0&3&123&393& 36&1240&11537&0&3&271&507& 197&3775&0&3&661&64\\
				&  & 250 &732&1903&212938&0&2&1440(1)&209& 1165&2503&235058&0&2&(2)&148& 3412(2)&15690&6.38&2&(2)&5\\
				&  & 500 & (3)&3216&440073&17.68&-&-& -& (3)&4457&305875&22.8&-&-&-&(3) &5760&24.01&-&-&-\\
				\hline		
	\end{tabular}}}
\end{sidewaystable}

In  \tabref{Table:IC_Rep}, we investigate the quality of the solution obtained by our proposed method for the instances with $|V|=120$ that are only solvable by the sampling-based approach. Because we do not have the true optimal solution, we cannot provide  exact deterministic optimality gaps. However, we use the approximate method proposed in \cite{Luedtke2008} to estimate the optimality gaps with statistical guarantees. In particular, we use Theorem 4 of \cite{Luedtke2008}, with $L=1, \alpha=\epsilon$. 
Let $\mathcal{M}$ be the number of replications of the sample approximation problems. In our computational study, for each choice of parameters in  \tabref{Table:IC-table2}, we have $\mathcal{M} = 3$ sample approximation problems. We obtain the optimal objective values of these $\mathcal{M} $ sample approximation problems solved by DCG-NV and report the minimum and maximum among these replications under the ``Master" column. Under the column ``Oracle," we  report the minimum and maximum objective function value of the feasible  solution provided by  the oracle phase. Note that if an objective value provided at the end of  the Benders phase is feasible, then the Benders and oracle phases share the same objective value for the corresponding instance. In \tabref{Table:IC_Rep}, we only report the instances that have at least two out of three sample approximation problems solvable by DCG-NV. (i.e., settings $(\bar b,\epsilon, |\Omega|) = (1,0.025,1000)$, $(1,0.05,1000)$ and $(100,0.05,500)$ are not reported). Column ``EGap" denotes the estimated gap that is equivalent to $(\overline{ub}-\overline{lb})/ \overline{ub}$, where $\overline {ub}$ is the best upper bound obtained by the oracle phase (i.e., Min value under Oracle), and  $\overline {lb}$ is the lower bound obtained from the Min value of Master. Luedtke and Ahmed \cite{Luedtke2008} show that for $|\Omega|\epsilon$ large enough so that a normal approximation to a binomial distribution is appropriate (e.g., $|\Omega|\epsilon\ge 5$), the lower bound obtained by taking the minimum of optimal objective function values of the sample approximation problems among the $\mathcal M$ replications   provides a valid lower bound with probability approximately $1-(0.5)^\mathcal M$. Using our method, we are now also able to give a deterministic upper bound on the optimal objective function value. Therefore, the gap reported under the EGap column is the estimated optimality gap with approximately 87.5\% confidence. We put ``*" on the EGap that for settings that do  not satisfy $|\Omega|\epsilon \ge 5$ in which case the approximate probabilistic guarantee  on the estimated gap is not valid.

From \tabref{Table:IC_Rep}, we observe that  Min/Max values under the  Master and Oracle columns are non-decreasing as the number of scenarios increases. For most of the instances with $(\bar b,|\Omega|) = (100,100)$, both minimum and maximum objective function  values  obtained at the end of the Benders phase are smaller than those obtained at the end of the oracle phase. For these cases, the Benders phase cannot provide even a feasible solution as can be seen from   \tabref{Table:IC-table2}, with Inf = 3. Note that even if  minimum objective function value  obtained at the end of the Benders phase is the same as that obtained at the end of the oracle phase,   the solution given by the Benders phase may not be feasible. For example, the setting $(\bar b,\epsilon, |\Omega|) = (1,0.05,100)$ has the  same objective function value in both phases. However, the corresponding instance in DCG-NV of \tabref{Table:IC-table2} has Inf = 3, which indicates that none of the solutions provided by the Benders phase is feasible. From EGap(\%), we observe that as we increase the sample size  $|\Omega|$, we obtain a  tighter gap between the feasible upper bound provided by the oracle phase and the lower bound of the optimal value provided by the Benders phase.  For the instances with unit costs and $|\Omega|\epsilon \ge 5$, we have a zero gap, which means that DCG-NV provides a truly optimal solution with probability at least 0.875. For non-unit cost instances with $|\Omega| \ge 250$, we are able to provide 1.75\%, 6.06\%, and 6.45\% EGaps with confidence approximately 0.875 for the settings  $(\bar b,\epsilon) = (100,0.0125)$, (100,0.025) and (120,100,0.05), respectively. Therefore, there is a trade-off between the scenario size, solution time and solution quality. If we use a small sample size, the solution time of the master problem is shorter. While in most cases, small sample size leads to infeasible solutions, because the resulting MIP is smaller, the oracle phase also takes a shorter time to fix the infeasibility. Overall, less time is spent in finding a feasible solution, however the quality of the solution may not be as good as that of a solution obtained by using a larger number of scenarios.

In the final part of this experimental subsection, we also test the influence of a larger size of $|V_2|$ compared to  $|V_1|$ in \tabref{Table:IC-table3}. To test instances with $|V_2|>|V_1|$, we follow the same experimental scheme as in \tabref{Table:IC-table1} but we increase $|V_2|$ from 30 to 60.  We observe that DCG can solve most instances within an hour for both Master and Oracle, and DCG-NV  outperforms DCG-Sub.   Compared to DCG, DEP \eqref{eq:sample-DEP} includes more constraints and additional $y_i^{\omega}$ variables for $i \in V_2$ and $\omega \in \Omega$. Therefore increasing $|V_2|$ increases the solution time of DEP significantly as can be seen by comparing the solution times in Tables \ref{Table:IC-table1} and \ref{Table:IC-table3}.  However, the sizes of the problems solved by DCG depend only on $|V_1|$, therefore the solution times of DCG are not sensitive to $|V_2|$. There is a slight increase in the time to check feasibility using the DP, however, this time is negligible.  Furthermore, we are also able to use the exact method, Oracle ($\kappa = 2$), to solve each instance optimally since the search space $|V_1|$ is still equal to 30. Hence, as in \tabref{Table:IC-table1}, we compare the optimality gap (oGap) between the optimal value given by DCG-NV and the optimal value given by the exact method for the instances in \tabref{Table:IC-table3}. The results show the same trend that the oracle phase  fixes the infeasibility of a solution obtained from a sample approximation, but the resulting feasible solution is not necessarily  optimal. For example,  for the non-unit cost instances we obtain solutions that have up to 8.54\% optimality gap with respect to the truly optimal solution.

In conclusion, for instances with small $|V_1|$, it is computationally tractable to obtain a truly optimal solution using the exact method Oracle ($\kappa = 2$). For larger $|V_1|$, an approximate sampling-based method is more effective than an exact method. Furthermore, solving the sample approximation problem with delayed cut generation (DCG) is more effective than solving the corresponding DEP model. The new valid inequalities enhance the performance of DCG when compared to using submodular inequalities. While  methods that rely on only solving a sample approximation problem cannot guarantee a feasible solution, our oracle-based method  combined with the statistical lower bounds of \cite{Luedtke2008} provides provably feasible solutions to the true problem with low optimality gaps with high confidence.  }

\revised{	
	
		\begin{table}[htb]
	{ \caption{ Solution quality  of DCG-NV for networks with $|V|=120$.}
		\label{Table:IC_Rep}
		\begin{center}
			\scalebox{0.95}{\begin{tabular}{ 			 |p{1cm}p{1cm}p{1cm}||p{1.5cm}p{1.5cm}||p{1.5cm}p{1.5cm}p{1.5cm}|}
					\hline
					
					& & &
					\multicolumn{2}{c||}{Master} &
					\multicolumn{3}{c|}{Oracle } \\
					
					\cline{4-8}%
					
					{ $\bar b$} & { $\epsilon$} & { $|\Omega|$}
					& {Min} & {Max }  & {Min }& { Max } & {EGap(\%)}
					\\
					\hline 		
					\multirow{7}{4em}{1}  & \multirow{3}{4em}{0.0125} & 100 & 5& 6 & 6  & 6 & $16.7^*$\\	
					&  & 500 & 6& 6 & 6  & 6 & 0\\	
					&  & 1000 & 6& 6 & 6  & 6 &  0\\							
					& \multirow{2}{4em}{0.025} & 100 & 5& 5 & 6  & 6 & $16.7^*$\\	
					&  & 500 & 6& 6 & 6 & 6 & 0\\

					& \multirow{3}{4em}{0.05} & 100 & 5& 5 & 5  & 5 & 0\\	
					&  & 500 & 5& 5 & 5 & 5 & 0\\	
					\hline

					\multirow{8}{4em}{100} & \multirow{3}{4em}{0.0125} & $100$ & 157& 160 & 171  & 172 & $8.19^*$\\	
					&  & $250$ & 163& 172 & 171  & 172 & $4.68^*$\\	
					&  & 500 & 168& 175 & 171  & 178 & 1.75\\							
					
					& \multirow{3}{4em}{0.025} & $100$ & 148  & 156 & 165& 167 &  $10.3^*$\\ 				
					&  & 250& 155 & 165 & 165  & 165  & 6.06\\
					&  & 500& 155 & 165 & 165  & 165  & 6.06\\

					& \multirow{2}{4em}{0.05} & 100 & 140  & 146 & 155& 155 &  9.68\\ 				
					&  & 250& 145 & 155 & 155  & 155  & 6.45\\	 					 				
					\hline%

			\end{tabular}}
	\end{center}}
\end{table}	

\begin{sidewaystable}[htp]
	{ \caption{Networks with $|V_1|=30$ and $|V_2|=60$ for PPSC with the independent probability coverage model.} 
		\label{Table:IC-table3}
		\centering 
		\scalebox{0.7}{\begin{tabular}{ |p{0.5cm}p{0.7cm}p{0.7cm}||p{1cm}p{1cm}p{1cm}p{1.2cm}|p{0.5cm}p{1cm}p{0.6cm}p{0.5cm}p{1.2cm}|p{1cm}p{1cm}p{1cm}p{1.2cm}|p{0.5cm}p{1cm}p{0.6cm}|p{1cm}p{1cm}p{1.2cm}|p{0.5cm}p{1cm}p{0.6cm}|  }
				\hline
				
				&	& & \multicolumn{9}{c|}{   DCG-NV } & \multicolumn{7}{c|}{   DCG-Sub } & \multicolumn{6}{c|}{ DEP \eqref{eq:sample-DEP} } 		
				\\		
				\cline{4-25}%
				&	& & \multicolumn{4}{c|}{   Master}& \multicolumn{5}{c|}{   Oracle } & \multicolumn{4}{c|}{  Master }& \multicolumn{3}{c|}{   Oracle } & \multicolumn{3}{c|}{  DEP }& \multicolumn{3}{c|}{  Oracle  } \\

				{$\bar b$} &	{$\epsilon$} & { $|\Omega|$}  
				& {Time(u) }& {Cuts} & {Nodes}  & {Gap(\%)} & {Inf} & {Time(u)} & {Cuts} & {nOpt} & {oGap(\%)}  &   {Time(u) }& {Cuts} & {Nodes}  & {Gap(\%)} & {Inf} & {Time(u)} & { Cuts } &   {Time(u) } & {Nodes}  & {Gap(\%)} & {Inf} & {Time(u)} & { Cuts }
				\\				
				\hline
				\multirow{9}{4em}{1}&\multirow{3}{4em}{0.0125}&100 & $\le 1$ &459 &265 &0 &3 &2 &130&0&0 &3 &1097 &892  &0 &3 &4 &68  & 23  &502 &0 &3 &80 &9  \\
				& & 500& 75 &2430 & 17033&0 &0 &- &-&0&0 &125 &4813 &19987 &0 &0 &-  &-  &  2775(2)  &5004 &34.7 &0 &- &-\\
				& & 1000 & 390 &4435 &51940 &0 &0 &- &-&0&0 &863 &8722 &73915 &0 &0 &-  &-   &  (3) &2142 &31.17 &- &- &-\\
				& \multirow{3}{4em}{0.025} & 100& 2 &492 &1056 &0 &3 &11 &157&0&0 &2 &937 &1274  &0 &3 &16  &134  &  16 &302 &0 &3 &191 &18 \\
				&  & 500&  291  &4322 &59958 &0 &0 &- &-&0&0 &584 &8832 &63429  &0 &0 &-  &-   & 3362(2)&9031 &28.94 &0 &- &-\\
				&  & 1000&  2684(1)  &8951 &294443 &29.42 &0 &- &-&0&0 &(3) &17066 &122811 &35.40 &- &-  &-  &  (3) &1784 &32.12 &- &- &-\\
				& \multirow{3}{4em}{0.05} & 100& 18 &824 &7867 &0 &3 &8 &13&0&0 &18 &1186 &10361  &0 &3 &13  &1  &50  &917 &0 &1 &103 &19\\
				&  & 500 & 618 &2822 &101165 &0 &3 &439 &18&0&0 &1064 &5266 &126920 &0 &2 &512  &35   & 936(1) &1260 &29.73 &0 &- &- \\
				&  & 1000 &795(1)& 5526&175280 &20 &0 &- &-&0&0&1350(1) &11248 &134195  &27.5 &2 &(2) &1  &  (3) &1226 &32.8 &- &- &- \\\cline{1-25}
				
				\multirow{9}{4em}{100}&\multirow{3}{4em}{0.0125}&100&  $\le 1$  &386 &280 &0 &3 &3 &139&1&1.19&2 &703 &546 &0 &3 &4 &139   & 8 &242 & 0&3 &65 &57\\
				& & 250 &  5 &800 &1700 &0 &3 &14 &59&0&0 &11 &1538 &3056 &0 &3 &16 &75   &151  &1538 &0 &3 &545 &20\\
				& & 500&  48 &1374 &9657 &0 &2 &49 &43&0&0 &85 &2644 &17354  &0 &2 &54  &63  & 1630  &4533 &0 &2 &(2) &6\\
				& \multirow{3}{4em}{0.025} & 100 &  $\le 1$  &347 &518 &0 &3 &2 &65&1&8.54 &$\le 1$ &770 &908 &0 &3 &3  &42   & 24 &699 &0 &3 &49 &11\\
				&  & 250 & 15   &873 &4835 &0 &2 &39 &34&0&0 &15 &1368 &5559 &0 &2 &26 &91  & 360  &2333 &0 &2 &756 &8\\
				&  & 500&  155  &1462 &36737 &0 &2 &95 &27&1&1.33 &275 &2880 &63236  &0 &2 &63  &57   & 1899(1) &5700 &7.26 &1 &(1) &2\\
				& \multirow{3}{4em}{0.05} & 100 & 3 &459 &2223 &0 &2 &6 &10&2&6.06 &4 &784 &2358 &0 &2 &7  &9   & 48 &1076 &0 &3 &84 &2\\
				&  & 250&  27 &865 &12720 & 0& 3 &49 &13&2&6.06 &51 &1387 &18840  &0 &3 &17  &25  & 735 &4538 &0 &3 &1650 &3\\
				&  & 500 &  511  & 1687& 116886&0 &2 &270 &7&1&6.06 &441 &3167 &107600  &0 &2 &167  &10   & (3) &5490 &14.52 &- &- &-\\
				\hline

	\end{tabular}}}
\end{sidewaystable}
}

\subsection{PPSC under the Linear Threshold Model}

In this subsection, we report our experiments with the linear threshold model. Given a complete bipartite graph, we assign a deterministic weight $a_{ij}$ to each arc $(i, j)$ from all nodes $ i \in V_1$ to all $j \in V_2$. We let \revised{$a_{ij}  = 0.9/|V_1^1| - i/(100  |V_1^1|)$} for each $i \in V_1^1$, and $a_{ij} = (\sum_{i=1}^{|V_1^1|}i/100)/|V_1^2|$ for each $i \in V_1^2$, which satisfies the requirement of the linear threshold model that $\sum_{i:(i,j)\in E}a_{ij}\le 1$.  Recall that in this model, each node $j\in V_2$ has a random threshold drawn from  a uniform distribution [0,1].   We let $ n = m = |V|/2 $, $|V_1^1| = 10$ for all instances, and $|V_1^2| = n - 10$. We consider risk levels \revised{$\epsilon \in \{0.0125,0.025,0.05\}$.}  We set the target $\tau  = 0.6m$.  \revised{The objective function coefficients are set as $b_i \in [1,\bar b]$ for each $i \in V_1$, where $\bar b \in \{1,100\}$. For $\bar b \neq 1$, we set $b_i =i\times \bar b/ |V_1^1|$ for $i \in V_1^1$, and $b_j =(|V_1|-j-1) \times \bar b/ (|V_1^2| \times 2)$ for $j \in V_1^2$.}

For the linear threshold model, $P(x,i) = \sum_{j \in V_1} a_{j,i}x_j$ is used in the DP oracle, $\mathcal{A}(x)$, where $a_{j,i}$ denotes a fixed weight on the arc $(j,i)$. This representation leads to an exact compact mixed-integer linear programming model \eqref{eq:PPSC-LTMIP}. To solve PPSC under the linear threshold model, we apply three methods. We first solve it exactly by using Algorithm \ref{alg:GDCG}, which is referred to as ``Oracle". We only study the cases of Oracle with $\kappa = 2$ in this subsection.  The second method is the deterministic equivalent problem  (DEP \eqref{eq:PPSC-LTMIP}) that uses the true distribution. We use the default setting of CPLEX with a single thread to solve DEP \eqref{eq:PPSC-LTMIP}. The dynamic programming formulation in DEP \eqref{eq:PPSC-LTMIP} computes the actual probability of covered nodes for a given selection from $V_1$ instead of sampling from the true distribution.  \revised{The third method is the sampling-based approach, where we take $|\Omega| = 1000$. We follow Propositions \ref{prop:sample_LT} and \ref{prop:sample_LT2}, and use formulation \eqref{eq:DEP_LT} (referred to as DEP-S \eqref{eq:DEP_LT}), to solve the sample approximation problem.  We summarize a comparison between these three methods in \tabref{Table:LT_oracle}. The proposed smaller formulation \eqref{eq:DEP_LT} has the best performance for our test instances.  Therefore, we no longer compare the performance of the original DCG-NV (or equivalently DCG-Sub) and the larger DEP \eqref{eq:sample-DEP} in $(x,y,z)$-space for the linear threshold models.  In the case of DEP-S \eqref{eq:DEP_LT}, once the sample approximation problem is solved to obtain a solution  $\bar x$,  we  check its feasibility  using the oracle phase. We add feasibility cuts as necessary until a feasible solution is reached. Note, again, that once a feasible solution is obtained at the end of the oracle phase, there is no guarantee that this solution is optimal to the true problem.}

\revised{
We observe that the exact method is only able to solve the instances with $|V| \le 70$. Hence, we only show the results of the instances with $|V| \in \{60,70\}$ in \tabref{Table:LT_oracle}. Column ``Time" denotes the total time of solving DEP-S \eqref{eq:DEP_LT} including the oracle phase. We do not report the solution time for both DEP-S \eqref{eq:DEP_LT} and the oracle phase separately since both phases can solve all instances extremely fast. Column ``Cuts" of DEP-S \eqref{eq:DEP_LT} denotes the number of feasibility cuts \eqref{eq:LLcut_strong} added to DEP-S \eqref{eq:DEP_LT} by the oracle phase. A positive value in Cuts indicates that the optimal solution given by DEP-S \eqref{eq:DEP_LT} is infeasible with respect to the original problem as detected by the oracle. In \tabref{Table:LT_oracle}, Oracle ($\kappa = 2$)  provides solutions for 2 out of 12 the instances. Compared to Oracle, our proposed compact reformulation (DEP \eqref{eq:PPSC-LTMIP}) solves 9 of the 12 instances optimally under the time limit. Note that both Oracle and DEP \eqref{eq:PPSC-LTMIP} solve all instances under the true distribution for this dataset. On the other hand, the sampling-based approach, DEP-S \eqref{eq:DEP_LT} combined with the oracle phase, provides an approximate solution  efficiently.   
 In column oGap, we compare the gap between the truly optimal objective value given by the exact method and the   objective value of the solution provided by DEP-S \eqref{eq:DEP_LT} after the oracle phase. We put ``-" in oGap if the exact method is not able to give the optimal solution within the time limit. 
 For our test instances for which a truly optimal solution is available, we see that the feasible solution found by the sampling-based approach is often optimal, there is only one setting  $(\bar b,\epsilon) = (100,0.0125)$, which has a 2.2\% optimality gap. In this case, the sampling-based approach cannot guarantee the optimality even we use a large number of scenarios such as $|\Omega| = 1000$. Furthermore, for the linear threshold instances with a large number of scenarios, most solutions to the sample approximation problem are feasible, there are only two settings ($(\bar b,\epsilon) = (1,0.0125)$ and (100,0.0125)) where oracle cuts were necessary to correct infeasibility of the solution given by the sample approximation.

\revised{
	\begin{table}[htb]
		{ \caption{The Exact and Sampling Methods for PPSC with the linear threshold model.}
			\label{Table:LT_oracle}
			\begin{center}
				\scalebox{0.75}{\begin{tabular}{ 			 |p{1cm}p{1cm}p{1cm}||p{1.5cm}p{1.5cm}||p{1.5cm}||p{1.2cm}p{1.2cm}p{1.2cm}p{1.2cm}p{1.2cm}|  }
						\hline
						
						& & & 
						\multicolumn{2}{c||}{Oracle ($\kappa = 2$)} & 
						\multicolumn{1}{c||}{DEP \eqref{eq:PPSC-LTMIP}} &
						\multicolumn{5}{c|}{DEP-S \eqref{eq:DEP_LT}}\\
						
						\cline{4-11}%
						
						$|V|$&{$\bar b$}  &  { $\epsilon$} 
						& {Time }& { Cuts }  & {Time } & {Time} & { Inf  } & { Cuts  } & { nOpt } & { oGap(\%) }
						\\
						\hline				
						\multirow{6}{4em}{60}&\multirow{3}{4em}{1} & 0.0125  &  2733 & 20224  & 437 & 2 & 3 & 4 & 0 &0 \\
						& & 0.025  &  1508 & 16083   & 440 &   2 & 0 & 0 & 0 &0\\
						& & 0.05   & $\ge 3600$ & 25203  & 668 & 7 & 0 & 0 & 0 &0  \\
						&\multirow{3}{4em}{100} & 0.0125   & $\ge 3600$ & 33174 & 2090 & 4 & 2 & 2 & 1& 2.2\\
						& & 0.025    &  $\ge 3600$ & 32163 & 1080 & 4 & 0 & 0 & 0 &0 \\			
						& & 0.05   & $\ge 3600$   &35864& 1478 &51 & 0 & 0 & 0 &0\\	
						\hline
						\multirow{6}{4em}{70}&\multirow{3}{4em}{1}  &  0.0125 & $\ge 3600$ & 27858  & 2022 &  $\le 1$ & 0 & 0 & 0 &0\\
						& & 0.025  & $\ge 3600$ & 26977  & 3453 &  2 & 0 & 0 & 0 &0\\
						& & 0.05   & $\ge 3600$ & 24881  & 3315 & 28 & 0 & 0 & 0 &0  \\
						&\multirow{3}{4em}{100} & 0.0125   & $\ge 3600$ & 37440 & $\ge 3600$ & 2 & 0 & 0 & - &- \\
						& & 0.025 &  $\ge 3600$ & 37165 & $\ge 3600$ & 7 & 0 & 0 & - &- \\			
						& & 0.05  & $\ge 3600$   &39837& $\ge 3600$ &55 & 0 & 0 & - &-\\				
						\hline%
				\end{tabular}}
		\end{center}}
	\end{table}
}

Next, we report our experience with the linear threshold instances with a larger size $|V|=120$  in \tabref{Table:LT_Rep}. We record the solution time for DEP-S \eqref{eq:DEP_LT} and the oracle phase  (referred to as Oracle), separately. We consider the number of scenarios with $|\Omega| \in \{100,500,1000\}$ and perform three replications of the sample set. For the oracle phase, as in the previous subsection, Inf records the number of infeasible solution detected by the oracle, and column Cuts records the number of inequalities \eqref{eq:LLcut_strong} added by the oracle phase. In our computational study, we observe that there are  instances that cannot reach a feasible solution during the oracle phase within the time limit. Hence, in Oracle, Time(u) records the solution time for those instances that completed this phase, and ``(u)" denotes the number of unsolved instances out of the Inf€ number of instances.  If Inf = 0, we put  ``-" in both Time(u) and Cuts. Since the exact method cannot provide the truly optimal solution within the time limit for the instances in \tabref{Table:LT_Rep}, we also provide the solution quality analysis in this part. For both DEP-S \eqref{eq:DEP_LT} and Oracle, we record the minimum and maximum objective function values  for the solution of each phase over the three replications, and the resulting estimated optimality gap.  Note that both DEP-S \eqref{eq:DEP_LT} and Oracle have the same objective value if an objective value provided by the DEP-S \eqref{eq:DEP_LT}  is feasible.

\revised{
	\begin{table}[htb]
		{ \caption{Solution analysis of DEP-S \eqref{eq:DEP_LT} for networks with $|V|=120$.}
			\label{Table:LT_Rep}
			\begin{center}
				\scalebox{0.75}{\begin{tabular}{ 			 |p{1cm}p{1cm}p{1cm}||p{1.5cm}p{1.5cm}p{1.5cm}||p{1.5cm}p{1.5cm}p{1.5cm}p{1.5cm}p{1.5cm}p{1.5cm}|}
						\hline
						
						& & &
						\multicolumn{3}{c||}{DEP-S \eqref{eq:DEP_LT}} &
						\multicolumn{6}{c|}{Oracle } \\
						
						\cline{4-12}%
						
						{ $\bar b$} & { $\epsilon$} & { $|\Omega|$}
						& {Time} & {Min} & {Max }  &  {Inf } &  {Time(u)} & {Cuts } & {Min }& { Max } & {EGap(\%)}
						\\
						\hline 		
						\multirow{9}{4em}{1}  & \multirow{3}{4em}{0.0125} & $100$ & $\le 1$& 9& 9 &0&-&- & 9  & 9 & 0\\	
						&  & 500 & $\le 1$ & 9& 9 &0&-&- & 9  & 9 & 0\\	
						&  & 1000 &3 & 9& 9&0&-&- & 9  & 9 &  0\\							
						& \multirow{3}{4em}{0.025} & 100 & $\le 1$ & 8& 9&1&$\le 1$&2 & 9  & 9 & $11.1^*$\\	
						&  & 500 & 2  & 9& 9&0&-&- & 9 & 9 & 0\\	
						&  & 1000 & 17  & 9& 9&0&-&- & 9  & 9 &  0\\							
						
						& \multirow{3}{4em}{0.05} & 100 & $\le 1$ & 8& 8&3&$\le 1$&7 & 9  & 9 & 11.1\\	
						&  & 500 & 25& 9& 9&0&-&- & 9 & 9 & 0\\	
						&  & 1000 & 163 & 9& 9&0&-&- & 9  & 9 &  0\\		 \hline

						\multirow{9}{4em}{100}  & \multirow{3}{4em}{0.0125} & $100$ & $\le 1$ & 391& 446&3&4(1)&16611 & 450  & 452 & $13.1^*$\\	
						&  & 500& 5 & 450& 450&0&-&- & 450  & 450 & 0\\	
						&  & 1000& 6 & 450& 450&0&-&- & 450  & 450 &  0\\							
						& \multirow{3}{4em}{0.025} & $100$& $\le 1$ & 376& 409&3&715(1)&24434 & 450  & 450 & $16.4^*$\\	
						&  & 500& 16 & 450& 450&0&-&- & 450 & 450 & 0\\	
						&  & 1000& 67 & 450& 450&0&-&- & 450  & 450 &  0\\							
						
						& \multirow{3}{4em}{0.05} & 100& $\le 1$ & 360& 369&3&14&1070 & 388  & 390 & 7.22\\	
						&  & 500& 7 & 374& 398&3&15&61 & 388 & 399 & 3.61\\	
						&  & 1000& 39 & 370& 404 &3&57&76 & 388  & 407 &  4.64\\					 					 				
						\hline%
				\end{tabular}}
		\end{center}}
	\end{table}	
}

From \tabref{Table:LT_Rep}, we see that although both Oracle and DEP \eqref{eq:PPSC-LTMIP} cannot solve any instance with $|V|=120$ within the time limit, DEP-S \eqref{eq:DEP_LT} can solve all instances well within a minute, on average. For a small number of scenarios $|\Omega|=100$, the solution time is within a second. However, a small number of scenarios may cause infeasible solutions. We observe that out of the six settings with $|\Omega|=100$, the oracle phase detects infeasibility in five settings (a total number of 13 instances under these settings is infeasible). In addition, for the setting $(\bar b, \epsilon)=(100,0.05)$, the oracle phase fixes the solutions of instances with a larger number of scenarios $|\Omega| \ge 500$. Considering the number of  cuts added and the solution time  in the oracle phase, we observe  that a small number of scenarios may also lead to a large number of oracle cuts, and few instances cannot even be solved during the oracle phase within the time limit (e.g., $(\bar b, \epsilon,|\Omega|) = (100,0.0125,100)$ and $(100,0.025,100)$).   Hence, a small sample size may be useful to solve the DEP faster, but more time may be spent to correct the resulting infeasible solutions.

Finally, we demonstrate the solution quality of the linear threshold instances with $|V|=120$. In \tabref{Table:LT_Rep}, a large value of EGap denotes that the objective value given by DEP-S \eqref{eq:DEP_LT} is far from the feasible objective value provided by the oracle. The observation shows that a small number of scenarios, in general, leads to a large estimated gap of EGap $> 10 \%$. In \tabref{Table:IC-table2}, we consider the condition that $|\Omega|\epsilon \ge 5$, and put ``*" on the EGap the settings that do not satisfy $|\Omega|\epsilon \ge 5$. The EGap without ``*" is valid with probability approximately 0.875 for the corresponding setting. For the instances with unit costs and $|\Omega|\epsilon \ge 5$, we have a zero gap, which means that DEP-S \eqref{eq:DEP_LT} provides a truly optimal solution with probability at least 0.875. Using $|\Omega| \ge 500$,  DEP-S \eqref{eq:DEP_LT} is also able to provide the truly optimal solution with confidence 0.875 for  most instances with $\bar b =100$.  In summary, for the linear threshold instances with $|\Omega| \ge 500$, DEP-S \eqref{eq:DEP_LT} not only solves the problem efficiently but also provides a high-quality solution with an estimated gap no larger than 5\%. 
}

\section{Conclusions and Future Work}\label{sec:con}

\revised{
In this paper, we propose a general delayed cut generation method to  solve  chance-constrained combinatorial optimization problems exactly (without sampling) when there is an efficient oracle to check whether a given solution satisfies the chance constraint. In addition, we show that the oracle can be used as a detector for checking the feasibility of the solution given by a sampling-based approach.  We demonstrate our proposed methods on a probabilistic partial set covering problem (PPSC) considered in the social networks literature, under certain probability distributions, one of which is finite but exponential (independent probability coverage) and the other is a continuous distribution (linear threshold). For the linear threshold formulation, we  give a compact MIP that linearly encodes the probability oracle within the optimization model. For  PPSC, we give strong valid inequalities for the deterministic equivalent formulation of the sample approximation problem and show that the proposed inequalities subsume the submodular inequalities that are valid for this problem.  In our computational study of the proposed methods, we observe that   the exact method is preferred for small networks. It provides provably optimal solutions with respect to the true distribution efficiently. However, we see that the 
sampling-based methods scale better when the size of the problem increases if they are able to exploit the problem structure. In particular, we show that using the proposed valid inequalities in a branch-and-bound framework enables the solution of problems with larger network sizes. While the optimal solution to the sample approximation problem may not even be feasible, our oracle-based method  can   check and correct the feasibility of the solution to obtain a high-quality feasible solution to the original problem. We note that our methods are generally applicable to other problems with the desired structure. For example, the probabilistic set covering problem with a circular distribution considered in \cite{Beraldi2002b,Luedtke2008} fits into our framework, although, in this case, we can also provide a compact MIP using the formulable structure of the probability oracle.

In this paper, we consider a class of CCPs with binary decision variables. A possible direction is to use the idea of oracles to solve other classes of CCPs exactly, such as those with continuous decision variables, in which case we are not able to use the no-good cuts. In addition, it will be useful to exploit the structure of the  problems to derive more effective feasibility cuts for the exact algorithm.    }

\section*{Acknowledgments}
We thank the three referees and the AE for their constructive comments that improved the paper. We also thank Baski Balasundaram for bringing \cite{Newman2003} to our attention.

\bibliographystyle{siam}
\bibliography{SIP-BIB}

\end{document}